\RequirePackage{fix-cm}
\documentclass[reqno]{amsart}

\usepackage{graphicx}
\usepackage[section]{placeins} 
\usepackage{url}
\usepackage[colorlinks=true,linkcolor=blue,citecolor=blue,pdfpagelabels=true]{hyperref}
\usepackage{doi}
\RequirePackage[capitalize,nameinlink]{cleveref}[0.19]

\Crefname{figure}{Figure}{Figures}
\Crefname{assumption}{Assumption}{Assumptions}

\crefformat{equation}{\textup{#2(#1)#3}}
\crefrangeformat{equation}{\textup{#3(#1)#4--#5(#2)#6}}
\crefmultiformat{equation}{\textup{#2(#1)#3}}{ and \textup{#2(#1)#3}}
{, \textup{#2(#1)#3}}{, and \textup{#2(#1)#3}}
\crefrangemultiformat{equation}{\textup{#3(#1)#4--#5(#2)#6}}%
{ and \textup{#3(#1)#4--#5(#2)#6}}{, \textup{#3(#1)#4--#5(#2)#6}}{, and \textup{#3(#1)#4--#5(#2)#6}}

\Crefformat{equation}{#2Equation~\textup{(#1)}#3}
\Crefrangeformat{equation}{Equations~\textup{#3(#1)#4--#5(#2)#6}}
\Crefmultiformat{equation}{Equations~\textup{#2(#1)#3}}{ and \textup{#2(#1)#3}}
{, \textup{#2(#1)#3}}{, and \textup{#2(#1)#3}}
\Crefrangemultiformat{equation}{Equations~\textup{#3(#1)#4--#5(#2)#6}}%
{ and \textup{#3(#1)#4--#5(#2)#6}}{, \textup{#3(#1)#4--#5(#2)#6}}{, and \textup{#3(#1)#4--#5(#2)#6}}

\usepackage{pgfplots}
\pgfplotsset{compat=newest}
\usetikzlibrary{arrows.meta, plotmarks, arrows.meta, pgfplots.fillbetween, shapes, math, patterns, angles, quotes}
\usepgfplotslibrary{patchplots}

\usepackage{grffile}
\usepackage{amsmath}
\usepackage{amsfonts}
\usepackage{bbm} 

\newtheorem{assumption}{Assumption}{\bf}{\rm}
{\it}{\rm}
\newtheorem{theorem}{Theorem}{\bf}{\rm}
{\bf}{\rm}
\newtheorem{remark}{Remark}{\rm}{\rm}

\crefname{assumption}{Assumption}{Assumptions}
\Crefname{assumption}{Assumption}{Assumptions}

\makeatletter 
\def\ifEmpty#1{\def\@temp{#1}\ifx\@temp\@empty} 
\makeatother

\newcommand{\R}{\mathbb{R}}
\newcommand{\C}{\mathbb{C}}

\newcommand{\cC}{\mathcal{C}}                                  
\newcommand{\cD}{\mathcal{D}}                                  
\newcommand{\cI}{\mathcal{I}}                                  
\newcommand{\cL}{\mathcal{L}}                                  
\newcommand{\cO}{\mathcal{O}}                                  
\newcommand{\tcC}{\tilde{\cC}}
\newcommand{\tcD}{\tilde{\cD}}

\DeclareMathOperator{\dif}{d \!}
\newcommand{\dr}{\dif{r}}
\newcommand{\ds}{\dif{s}}
\newcommand{\dt}{\dif{t}}

\newcommand{\ddt}{\frac{\dif{}}{\dif{t}}}

\renewcommand{\exp}[1]{\mathrm{e}^{#1}}
\newcommand{\eA}[1]{\exp{#1 A}}

\newcommand{\set}[2]{\{ #1 \;;\; #2 \}}
\newcommand{\domain}[1]{\mathcal{D}(#1)}
\newcommand{\rank}[1]{\operatorname{rank}\!\left(#1 \right)}

\newcommand{\nrm}[2][]{\left\lVert #2 \right\rVert\ifEmpty{#1}\else_{#1}\fi}
\newcommand{\iprod}[2]{\left( #1, #2 \right)}

\newcommand{\LHH}{\cL(H)}

\begin{document}

\title[Singular value decay of operator-valued equations]{Singular value decay of operator-valued differential Lyapunov and Riccati equations}

\author{Tony Stillfjord}
\address{Max Planck Institute for Dynamics of Complex Technical Systems, Sandtorstr.\ 1, DE-39106 Magdeburg, Germany}
\curraddr{}
\email{stillfjord@mpi-magdeburg.mpg.de}
\thanks{}






\date{Received: date / Accepted: date}

\begin{abstract}
We consider operator-valued differential Lyapunov and Riccati equations, where the operators $B$ and $C$ may be relatively unbounded with respect to $A$ (in the standard notation). In this setting, we prove that the singular values of the solutions decay fast under certain conditions. In fact, the decay is exponential in the negative square root if $A$ generates an analytic semigroup and the range of $C$ has finite dimension. This extends previous similar results for algebraic equations to the differential case. When the initial condition is zero, we also show that the singular values converge to zero as time goes to zero, with a certain rate that depends on the degree of unboundedness of $C$.
A fast decay of the singular values corresponds to a low numerical rank, which is a critical feature in large-scale applications. The results reported here provide a theoretical foundation for the observation that, in practice, a low-rank factorization usually exists.
\end{abstract}

\maketitle

\section{Introduction} \label{sec:introduction}
We consider differential Lyapunov equations (DLEs) and differential Riccati equations (DREs) of the forms
\begin{equation}
  \label{eq:DLE}
  \dot{P} = A^* P + PA + C^*C, \quad P(0) = G^*G,
\end{equation}
and
\begin{equation}
  \label{eq:DRE}
  \dot{P} = A^* P + PA + C^*C - P BB^* P, \quad P(0) = G^*G,
\end{equation}
respectively. Such equations arise in many different areas, e.g.\ in optimal/robust control, optimal filtering, spectral factorizations, $\mathbf{H}_\infty$-control, differential games, etc.~\cite{AboFIJ03,BasB95,IchK99,PetUS00}.

A typical application for DREs is a linear quadratic regulator (LQR) problem, where one seeks to control the output $y = Cx$ given the state equation $\dot{x} = Ax + Bu$ by varying the input $u$. In the case of a finite time cost function, 
\begin{equation*}
  J(u) = \int_{0}^{T}{ \nrm{y(t)}^2 + \nrm{u(t)}^2 \dt} + \nrm{Gx(T)}^2, 
\end{equation*}
it is well known that the optimal input function $u^{\text{opt}}$ is given in state feedback form. In particular, $u^{\text{opt}}(t) = -B^*P(T-t)x(t)$, where $P$ is the solution to the DRE \cref{eq:DRE}~\cite{BenLP08,LasT00}.

The solution to the DLE, on the other hand, yields the (time-limited) observability Gramian of the corresponding LQR system. It is used in applications such as model order reduction~\cite{morBauBF14,morGawJ90} for determining which states $x$ have negligible effect on the input-output relation $u \mapsto y$, and which can therefore safely be discarded from the system~\cite{morKue17,morBenKS16}.

In the continuous case, the equations \cref{eq:DLE}, \cref{eq:DRE} are operator-valued. After a spatial discretization they become matrix-valued. Approximating their solutions by numerical computations is thus, if done naively, much more expensive than simply approximating, e.g., the corresponding vector-valued equation $\dot{x} = Ax$. A standard way to decrease the computational complexity is to utilize structural properties of the solutions. A commonly used such property is that of low numerical rank~\cite{LiW02,LanMS15,Sti15}, i.e.\ a fast (often exponential) decay of the singular values. This allows us to approximate $P(t) \approx L(t) L(t)^*$ where $L(t)$ is of finite rank. In the matrix-valued setting, we would have $P(t) \in \R^{n \times n}$ and $L(t) \in \R^{n \times r}$ with $r \ll n$.

While there exist results on when such low numerical rank is to be expected for \emph{algebraic} Lyapunov and Riccati equations (i.e.\ the stationary counterparts of \cref{eq:DLE,eq:DRE}), see e.g.~\cite{AntSZ02,SorZ02,BakES15,BenB13,Pen00,BenB16,GruK14,Opm15}, the differential case has so far been neglected in the literature.

The aim of this article is to remedy this situation and provide criteria on $A$, $B$ and $C$ that guarantee a certain decay of the singular values $\{\sigma_k\}_{k=1}^{\infty}$ of the solutions to \cref{eq:DLE,eq:DRE}. 
We consider the operator-valued case, with the standard assumption that $A$ generates an analytic semigroup. In the LQR setting, this corresponds to the control of abstract parabolic problems (including, for example, heat flows and wave equations with strong damping). We allow relatively unbounded operators $B$ and $C$, which means that we can treat various forms of boundary control and observation. In this setting, we follow the approach suggested in~\cite{Opm15} for algebraic equations. There, a decay of the form $\sigma_k \le M\exp{-\gamma \sqrt{k}}$ was shown, i.e.\ we can not expect exponential decay but only exponential in the square root.
The main results of the present article demonstrates that this extends to the differential case, under similar assumptions. In the case that $G = 0$ (and hence $P(0) = 0$), our bounds additionally show that the singular values converge to $0$ as $t \to 0$ with a rate $t^{1-2\alpha}$, where $\alpha$ is a measure of how unbounded the output operator $C$ is.

An outline of the article is as follows: In \cref{sec:preliminaries} we specify the abstract framework, state the assumptions on the operators and recall some resulting properties of the solutions to \cref{eq:DLE,eq:DRE}. Then in \cref{sec:DLE} we use the concept of sinc quadrature to show that certain finite-rank operators approximate the integral $\int_{0}^{t}{\iprod{C\exp{sA} \cdot}{C\exp{sA}\cdot} \ds}$ well. Since this is in fact the solution to \cref{eq:DLE} when $G = 0$, the main results for DLEs then follow quickly. We generalize these results to DREs in \cref{sec:DRE} by factorizing the system using output and input-output mappings. Finally, in \cref{sec:experiments}, we perform a number of numerical experiments on discretized versions of the equations, which verify the theoretical statements.

\section{Preliminaries} \label{sec:preliminaries}

In the operator-valued case, \cref{eq:DLE,eq:DRE} need to be interpreted in an appropriate sense. Here, we mainly follow~\cite{LasT00} (see also~\cite{BenDDetal07}), and outline the ideas for the DRE \cref{eq:DRE} since all the results carry over to the DLE \cref{eq:DLE} by setting $B=0$. Thus, let $H$, $Y$, $U$ and $Z$ be Hilbert spaces, and let the following operators be given: the (unbounded) state operator $A : \domain{A} \subset H \to H$, the input operator $B : U \to \domain{A^*}'$, the output operator $C : \domain{A} \to Y$ and the final state penalization operator $G : H \to Z$. This corresponds to problems arising from the linear quadratic regulator setting.

By $A^*$ we mean the adjoint of $A$ with respect to the inner product on $H$, and $\domain{A^*}'$ denotes the dual space of $\domain{A^*}$, also with respect to the $H$-topology. With the proper interpretation (see e.g.~\cite{LasT00}), it is a superset of $H$; in fact, the completion of $H$ in the norm $\nrm{A^{-1}\cdot}_H$. Additionally, for general Hilbert spaces $X$ and $Y$ we use the notation $\cL(X, Y)$ to denote the set of linear bounded operators from $X$ to $Y$.
\begin{remark}
  In order that the notation conforms to the usual evolution equation setting, we have changed the direction of time so that $P(0) = G^*G$ is the given condition rather than $P(T) = G^*G$ as in~\cite{LasT00}. The only effect of this is to change the signs of all the terms on the right-hand-side.
\end{remark}
Our main assumption is
\begin{assumption}\label{assumption:A}
  The operator $A: \domain{A} \subset H \to H$ is the generator of a strongly continuous analytic semigroup $\exp{tA}$ on $H$.
\end{assumption}
This means that there exists a $\delta \in (0, \pi/2]$ such that $z \mapsto \exp{zA}$ is analytic on the sector $\Delta_{\delta} = \set{z \in \C}{z \neq 0, |\arg(z)| < \delta}$. Further, there exist constants $\omega \in \R$ and $M \ge 0$ such that the fractional powers $(\omega I - A)^{\gamma}$ are well defined, and we have the inequalities $\nrm{\exp{tA}} \le M\exp{\omega t}$ and $\nrm{(\omega I - A)^{\gamma}\exp{tA}} \le M(1 + t^{-\gamma})\exp{\omega t}$, see e.g.~\cite[Section 3.10]{Sta05}. Here, $\omega <  0$ corresponds to the stable case, but we allow $\omega > 0$ too. We also note that $A^*$ is the generator of $\exp{tA^*} = (\exp{tA})^*$.

Further, we allow both $B$ and $C$ to be unbounded operators, but not \emph{too} unbounded. In particular, 
\begin{assumption}\label{assumption:B}
  The operator $B : U \to \domain{A^*}'$ is relatively bounded in the sense that there is a $\beta \in [0, 1)$ such that $(\omega I -A)^{-\beta}B \in \cL(U, H)$.
\end{assumption}
\begin{assumption}\label{assumption:C}
  The operator $C : \domain{(\omega I -A)^{\alpha}} \to Y$ is relatively bounded in the sense that $C(\omega I - A)^{-\alpha} \in \cL(H, Y)$ for $0 \le \alpha < \min(1-\beta, 1/2)$, with the parameter $\beta$ from \cref{assumption:B}.
\end{assumption}
Finally, $G$ needs to provide sufficient smoothing to compensate for the roughness of $B$:
\begin{assumption}\label{assumption:G}
  The operator $G : H \to Z$ is bounded. If $\beta \ge 1/2$, there should also exist a $\theta \ge \beta - 1/2$ such that $G(\omega I -A)^{\theta} : H \to Z$.
\end{assumption}
\begin{remark}
  In the DLE case, we have $B = 0$. \cref{assumption:B} is thus always satisfied and there is no extra restriction on $\alpha$ in \cref{assumption:C} except $\alpha \in [0, 1/2)$.
\end{remark}

\begin{remark}
  \Cref{assumption:G} is marginally stronger than the assumption that $(\omega I - A^*)^{\theta} G^*G \in \LHH$, $\theta > 2\beta -1$, which is made in~\cite{LasT00}. We use \cref{assumption:G} for compatibility with results from the Salamon/Weiss/Staffans framework~\cite{Sta05}, but it can most likely be weakened to the one in~\cite{LasT00}.
\end{remark}

Under \cref{assumption:A,assumption:B,assumption:C,assumption:G}, the DRE \cref{eq:DRE} possesses a classical solution $t \mapsto P(t) \in \LHH$, see e.g.~\cite[Theorem 1.2.2.1]{LasT00}. This solution additionally solves the following integral equation for all $x, y \in H$, and vice versa:
\begin{equation}
\begin{aligned} \label{eq:DRE_int}
  \iprod{P(t)x}{y} &= \iprod{G\exp{tA}x}{G\exp{tA}y} + \int_{0}^{t}{\iprod{C\exp{sA}x}{C\exp{sA}y} \ds} \\
 &- \int_{0}^{t}{\iprod{B^*P(s)\exp{sA}x}{B^*P(s)\exp{sA}y} \ds} .
\end{aligned}
\end{equation}

Combining \cref{assumption:A} and \cref{assumption:C} shows that $C\exp{sA} \in \cL(H, Y)$ for $s > 0$. This actually holds on every subset of the sector of analyticity $\Delta_{\delta}$, as demonstrated e.g.\ in~\cite{Opm15}.
In particular, for every $a \in [0,1)$ there exist positive constants $M_a$ and $\omega$ such that 
$\nrm{C\exp{zA}}_{\cL(H,Y)} \le M_a(1 + |z|^{-\alpha}) \exp{\omega \Re(z)}$ for all $z \in \Delta_{a\delta}$. The constants $M_a$ go to infinity as $a \to 1$, i.e.\ as we approach the limit of analyticity. However, by simply redefining $\delta$ as, e.g., $\delta/2$ we can always get a uniform estimate. In the following, we will therefore omit the dependence on $a$ and write 
\begin{equation}  \label{eq:Cexp_bound}
  \nrm{C\exp{zA}}_{\cL(H,Y)} \le M(1 + |z|^{-\alpha})\exp{\omega \Re(z)}, \quad z \in \Delta_{\delta},
\end{equation}
for two positive constants $M$ and $\omega$.
Since $\alpha < 1/2$, $|z|^{-2\alpha}$ is integrable at $0$ and the first integral term in \cref{eq:DRE_int} is therefore well-defined. That the second integral term is well-defined under \cref{assumption:A,assumption:B,assumption:C,assumption:G} is less straightforward, due to the presence of $P(s)$ and the fact that $\beta$ is allowed to take values in $[1/2, 1)$. We refer to~\cite[Chapter 1]{LasT00}.

\section{Lyapunov equations} \label{sec:DLE}
Let us first consider the Lyapunov case \cref{eq:DLE}. Restricting \cref{eq:DRE_int} by setting $B = 0$ shows that 
\begin{equation} \label{eq:DLE_int}
\iprod{P(t)x}{y} = \iprod{G\exp{tA}x}{G\exp{tA}y} + \int_{0}^{t}{\iprod{C\exp{sA}x}{C\exp{sA}y} \ds},
\end{equation}
which provides a closed-form expression for the solution $P$.
For $x, y \in \domain{A}$ we denote the integrand by $F$;
\begin{equation}
  \label{eq:F_def}
  F(z) = \iprod{ C \eA{z}x}{ C \eA{z} y},
\end{equation}
and note that in fact $F : \Delta_\delta \to \C$. By \cref{eq:Cexp_bound}, for all $x,y \in \domain{A}$ we have the bound
\begin{equation} \label{eq:F_bound}
|F(z)| \le \frac{M^2}{|z|^{2\alpha}}\exp{2\omega \Re(z)} \nrm{x}\nrm{y}.
\end{equation}

Our aim is now to approximate the integral $\int_{0}^{t}{F(s)\ds}$ by sinc quadrature, which converges exponentially in the number of quadrature nodes. The basic idea is to map the interval $(0,t)$ onto the real line, apply the trapezoidal rule, use decay properties of $F$ at $\pm \infty$ and then transform back. The proof uses complex analysis and thus requires us to consider $(0,t)$ as a subset of a domain in $\C$ rather than a real interval. In our case, the appropriate mapping is $\phi_t \colon \C \to \C$, $\phi_t(z) = \ln \frac{z}{t-z}$, with inverse $\psi_t \colon \C \to \C$, $\psi_t(w) = \frac{t\exp{w}}{\exp{w} + 1}$. The function $\phi_t$ maps the eye-shaped domain
\begin{equation*}
  D_E^d(t) = \set{z \in \C}{|\arg \Big(\frac{z}{t-z}\Big)|  < d },
\end{equation*}
where $0 < d < \pi/2$, onto the infinite strip
\begin{equation*}
  D_S^d(t) = \set{w \in \C}{|\Im w| < d }.
\end{equation*}
Here, of course, $D_E^d(t) \supset [0,t]$. See \cref{fig:transformation} for an illustration of these domains.

\begin{figure}
  \centering
  \begin{tikzpicture}
    [
    declare function={
      m(\x,\t,\y) = \t / (exp(-\x) + exp(\x) + 2*cos(deg(\y)));
      n(\x,\t,\y) = \t / (exp(-2*\x) + 1 + 2*exp(-\x)*cos(deg(\y)));
      re(\x,\t,\y) = m(\x,\t,\y) * cos(deg(\y)) + n(\x,\t,\y);
      im(\x,\t,\y) = m(\x,\t,\y) * sin(deg(\y));
    },
    ]
    \tikzmath{\d = -pi/2*0.6; \sind = sin(deg(\d)); \cosd = cos(deg(\d));}

    \begin{axis}[name=eye,
      xshift=0cm,
      xmin=-1, xmax=3, 
      ymin=-2, ymax=2, 
      axis x line=middle,
      axis y line=middle,
      axis line style={->,},
      xlabel={$z_1$},
      xlabel style = {at = (current axis.right of origin), anchor = north west},
      ylabel={$z_2$},
      ylabel style = {left},
      xtick = {0,2},
      xticklabels = {0, $t$ },
      ytick = {0,1},
      yticklabels = {0,1},
      scale = 0.75,
      , after end axis/.append code={\coordinate (eyea) at (axis cs:2,1.5); \coordinate (eyeb) at (axis cs:2,-1.5);},
      ]

      \addplot [name path=coneup, parametric, blue, domain=0:1.5,  variable=\x]
      ({x*\cosd}, {-x*\sind});

      \addplot [name path=coneup, parametric, blue, domain=0:1.5,  variable=\x]
      ({2 - x*\cosd}, {x*\sind});

      \coordinate (a) at (0,0);
      \coordinate (b) at (1,0);
      \coordinate (c) at (\cosd,-\sind);
      \coordinate (d) at (2,0);
      \coordinate (e) at (2-\cosd,\sind);

      \draw pic[draw, angle radius=1cm,"$d$" shift={(4mm, 5mm)}] {angle=b--a--c};
      \draw pic[draw, angle radius=1cm,"$d$" shift={(-4mm, -5mm)}] {angle=b--d--e};

      \addplot [name path=upperarc, parametric, blue, thick, domain = -10:10, samples=100, variable=\x]
      ({re(x,2,\d)}, {im(x,2,\d)});
      \addplot [name path=lowerarc, parametric, blue, thick, domain = -10:10, samples=100, variable=\x]
      ({re(x,2,\d)}, {-im(x,2,\d)});

      \addplot [name path=eyeinterval,  red, thick, domain = 0:2, variable=\x]
      {0};

      \tikzfillbetween[of=upperarc and lowerarc]{pattern=north east lines}

    \end{axis}   

    \begin{axis}[name=strip,
      xshift=6cm,
      xmin=-10, xmax=10, 
      ymin=-2, ymax=2, 
      axis x line=middle,
      axis y line=middle,
      axis line style={->,},
      xlabel={$w_1$},        
      xlabel style = {at = (current axis.right of origin), anchor = north west},
      ylabel={$w_2$},
      ylabel style = {left},
      xtick = {0},
      xticklabels = {,, },
      ytick = {-0.95, 0, 0.95},
      yticklabels = {\raisebox{-1em}[0pt][0pt]{$-d$}, 0, \raisebox{.5em}[0pt][0pt]{$d$}},
      scale = 0.75,
      , after end axis/.append code={\coordinate (stripa) at (axis cs:-4,1.5); \coordinate (stripb) at (axis cs:-4,-1.5);},
      ]
      \addplot[name path=upperline, blue, thick, domain = -8:8, samples=100]
      {pi/2*0.6};
      \addplot [name path=lowerline, parametric, blue, thick, domain = -8:8, samples=100]
      {-pi/2*0.6};

      \addplot [name path=stripinterval, red, thick, domain = -8:8, samples=100, variable=\x]
      {0};

      \draw[color=black, ->] (8.5,-0.5) to (9.5,-0.5);
      \draw[color=black, ->] (8.5,0.5) to (9.5,0.5);

      \draw[color=black, ->] (-8.5,-0.5) to (-9.5,-0.5);
      \draw[color=black, ->] (-8.5,0.5) to (-9.5,0.5);

      \tikzfillbetween[of=upperline and lowerline]{pattern=north east lines}

    \end{axis}   

    \draw[color=black, ->, thick]
    (eyea) to[bend left] node[midway,above] {$\phi_t(z) = \ln \frac{z}{t-z}$} (stripa);
    \draw[color=black, ->, thick] (stripb) to[bend left] node[midway,below] {$\psi_t(w) = \frac{t\exp{w}}{\exp{w}+1}$} (eyeb);

  \end{tikzpicture}
  \caption{The transformations $\phi_t$, $\psi_t$ and the domains $D_E^d(t)$ (shaded, left), $D_S^d(t)$ (shaded, right).}
  \label{fig:transformation}
\end{figure}
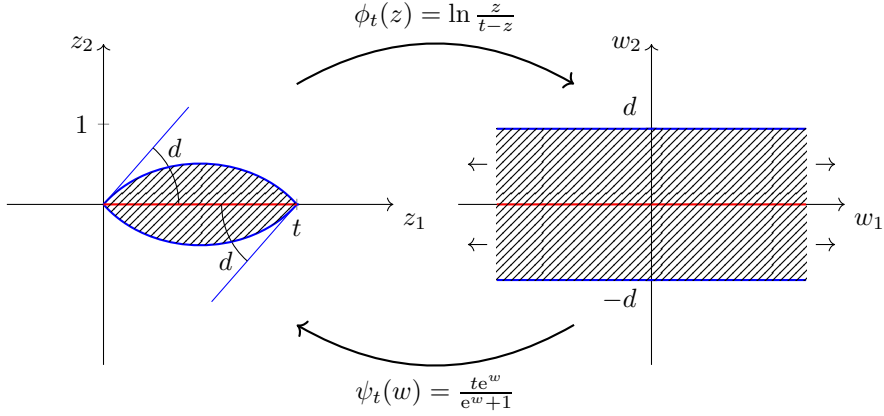

The following result is due to Lund and Bowers~\cite{LunB92}, inspired by~\cite{Ste73}. Here, as well as throughout the rest of the paper, we use the letter $M$ to denote a generic constant that does not depend on $t$. It is not necessarily the same $M$ as in \cref{eq:F_bound,eq:Cexp_bound}.
\begin{theorem}[{\cite[Theorem 3.8]{LunB92}}] \label{thm:LB}
Let $f$ be an analytic function on $D_E^d(t)$ that for some $r \in (0,1)$ satisfies the condition
\begin{equation}
  \label{eq:LBcond1}
  \int_{\psi_t(u + L)} { |f(z)| \dif{z} } = \cO(| u |^r), \quad u \to \pm \infty,
\end{equation}
where $L = \set{iv}{|v| \le d}$. Further assume that 
\begin{equation}
  \label{eq:LBcond2}
  B(f) := \lim_{\gamma \to \partial D_E^d(t)}{ \int_{\gamma}{ |f(z)| \dif{z} } < \infty },
\end{equation}
  where $\gamma$ denotes any closed simple contour in $D_E^d(t)$, and that there are positive constants $M$, $\rho$ and $\mu$ such that
  \begin{equation}
    \label{eq:LBcond3}
    \bigg| \frac{f(z)}{\phi_t'(z)} \bigg| \le M
    \begin{cases}
      \exp{-\rho |\phi_t(z) | } \quad \forall z \in \psi_t\big((-\infty, 0)\big)\\
      \exp{-\mu |\phi_t(z) | } \quad \forall z \in \psi_t\big([0, -\infty)\big)
    \end{cases}.
  \end{equation}
 Choose
  \begin{equation*}
    n = \Big\lceil{\frac{\rho}{\mu} m + 1} \Big\rceil, \quad h = \Bigg( \frac{2 \pi d}{\rho m} \Bigg)^{1/2},
  \end{equation*}
  with $m$ a nonnegative integer large enough that $h \le \frac{2 \pi d}{\ln 2}$, and define the quadrature nodes $z_k$ and weights $w_k$ by
  \begin{equation*}
    z_k = \psi_t(kh) = \frac{t\exp{kh}}{\exp{kh} + 1},
    \quad  w_k = \Big( \phi_t'(z_k) \Big)^{-1} = \frac{t \exp{kh}}{(\exp{kh} + 1)^2}.
  \end{equation*}
  Then it holds that
  \begin{equation*}
    \Bigg| \int_{0}^{t}{f(z) \dif{z}} - h\sum_{k=-m}^{n}{w_k f(z_k)} \Bigg| 
    \le \bigg(\frac{M}{\rho} + \frac{M}{\mu} + 2B(f) \bigg) \exp{-(2 \pi \rho d m)^{1/2}}.
  \end{equation*}
\end{theorem}
Specifying this theorem to the function $F$ given in \cref{eq:F_def} leads to
\begin{theorem}\label{thm:int_approx}
Let \cref{assumption:A,assumption:C}  be satisfied, and let $h$, $n$, $z_k$ and $w_k$ be chosen as in \cref{thm:LB} with $d = \delta$. Then there is a positive constant $M$, independent of $t$, $x$ and $y$, but dependent on $\alpha$, such that
\begin{equation*}
  \Bigg| \int_{0}^{t}{F(z) \dif{z}} - h\sum_{k=-m}^{n}{w_kF(z_k)} \Bigg| 
    \le M t^{1-2\alpha} \exp{-(2 \pi (1-2\alpha) \delta m)^{1/2}} \nrm{x} \nrm{y}.
\end{equation*}
\end{theorem}
\begin{proof}
  We verify the conditions of \cref{thm:LB}. Since the domain $D_E^\delta(t)$ is a subset of the cone $\set{w \in \C}{|\arg{w}| \le \delta}$ for any $t > 0$, the function $F$ is clearly analytic on $D_E^\delta(t)$. Suppose that $z = \psi_t(u+iv)$ where $|v| \le \delta$. Then
  \begin{equation*}
    \Big|\frac{\dif{z}}{\dif{v}}\Big| = \frac{t\exp{u}}{|\exp{u}\exp{iv} + 1|^2} \le t\min(\exp{u}, \exp{-u}) \le t,
  \end{equation*}
  since $\delta < \pi / 2$ means that $|\exp{u}\exp{iv} + 1| \ge \max(1, \exp{u})$. Hence
  \begin{align*}
    \int_{\psi_t(u + L)} { |F(z)| \dif{z} } &\le \int_{-\delta}^{\delta}{ \bigg|F\bigg(\frac{t\exp{u}\exp{iv}}{\exp{u}\exp{iv} + 1}  \bigg)  \bigg| t \, \dif{v} } \\
                                         &\le M t  \int_{-\delta}^{\delta}{ \bigg|\frac{t\exp{u}\exp{iv}}{\exp{u}\exp{iv} + 1} \bigg|^{-2\alpha} \dif{v} } \\
                                         &\le 2 M \pi t^{1-2\alpha},
  \end{align*}
where we have used \cref{eq:F_bound} as well as the estimate $\exp{2\omega \Re(z)} \le \max(1, \exp{2\omega T}) \le M$ in the second step and the inequality $|\exp{u}\exp{iv} + 1| \le \exp{u}+1 \le 2 \exp{u}$ in the third step. As this bound is independent of $u$ and $1 - 2\alpha > 0$ due to \cref{assumption:C}, the first condition \cref{eq:LBcond1} is satisfied.

To check the second condition, we make a change of variables $w = \eta(z) = \frac{z}{t-z}$. It is easily seen that $\eta$ maps the boundary of $D_E^{\delta}(t)$ onto the rays $\set{r\exp{\pm i\delta}}{r \ge 0}$, that the inverse is given by $z = \eta^{-1}(w) = \frac{tw}{1+w}$ and that the derivative of the inverse is given by $w \mapsto \frac{t}{(1+w)^2}$. Denoting the top and bottom parts of $\partial D_E^{\delta}(t)$ by $\partial D_{+}$ and $\partial D_{-}$, respectively, we thus have $B(F) = \int_{\partial D_{+}}{|F(z)\dif{z}|} + \int_{\partial D_{-}}{|F(z)\dif{z}|}$ where
\begin{align*}
  \int_{\partial D_{\pm}}{|F(z)|\dif{z}} &= \int_{0}^{\infty}{ \bigg|F\bigg(\frac{t r\exp{\pm i\delta}}{1 + r\exp{\pm i\delta}}\bigg)\bigg| t \big|1 + r\exp{\pm i\delta} \big|^{-2} \dr} \\
                                   &\le M\int_{0}^{\infty}{ \bigg|\frac{t r\exp{\pm i\delta}}{1 + r\exp{\pm i\delta}}\bigg|^{-2\alpha} t \big|1 + r\exp{\pm i\delta} \big|^{-2} \dr},
\end{align*}
again using \cref{eq:F_bound} and bounding the exponential term by $\max(1, \exp{2\omega T}) $. As $|1 + r\exp{\pm i\delta} | \ge \max(1, r)$ we get
\begin{equation*}
  \int_{\partial D_{\pm}}{|F(z)|\dif{z}} \le t^{1-2\alpha} \Bigg( \int_{0}^{1}{ r^{-2\alpha} \dr} + \int_{1}^{\infty}{ r^{-2} \dr} \Bigg),
\end{equation*}
so that, in conclusion,
\begin{equation*}
  B(F) \le 2 t^{1-2\alpha} \Big( \frac{1}{1-2\alpha} + 1 \Big).
\end{equation*}

Finally, we check condition \cref{eq:LBcond3}. A simple computation shows that $\phi_t'(z) = \frac{t}{z(t-z)}$. Clearly, $\psi_t\big((-\infty, 0)\big) = (0, t/2) =: \Gamma_1$ and $\psi_t\big([0, \infty)\big) = [t/2, t) =: \Gamma_2$, which means that on these intervals we have
\begin{equation*}
  \exp{-\rho |\phi_t(z)|} = z^{\rho} (t-z)^{-\rho} \quad \text{and} \quad \exp{-\mu |\phi_t(z)|} = z^{-\mu} (t-z)^{\mu}.
\end{equation*}
On $\Gamma_1$, $|t-z| \le t$, so by \cref{eq:F_bound} we get
\begin{align*}
  \bigg| \frac{F(z)}{\phi_t'(z)} \bigg| &\le M |z|^{-2\alpha}\exp{2\omega\Re(z)} |z||t-z| t^{-1} 
                                      \le M |z|^{1-2\alpha} t^{-1} |t-z|^{2\alpha -1} |t-z|^{2-2\alpha} \\
                                      &\le M t^{1-2\alpha} |z|^{1-2\alpha} |t-z|^{2\alpha -1},
\end{align*}
i.e.\ the desired bound holds with $\rho = 1 - 2\alpha$ and constant $M t^{1-2\alpha}$, where $M$ is independent of $t$.
On $\Gamma_2$, $|z| \le t$, and we similarly get
\begin{align*}
  \bigg| \frac{F(z)}{\phi_t'(z)} \bigg| &\le M |z|^{1-2\alpha} |t-z| t^{-1} 
                                       \le M |z|^{-1}|t-z| |z|^{2-2\alpha} t^{-1} \\
                                      &\le M t^{1-2\alpha} |z|^{-1}|t-z|,
\end{align*}
i.e.\ the desired bound holds with $\mu = 1$ and constant $M t^{1-2\alpha}$, where $M$ is again independent of $t$.
\end{proof}

We denote the singular values of $P$ by $\sigma_k(P)$ and order them in decreasing order. Let us first consider the case when $G = 0$.
\begin{theorem} \label{thm:DLE_decay_G0}
 Let \cref{assumption:A,assumption:C}  be satisfied, with the output space $Y$ having finite dimension $\dim Y \ge 1$. Further assume that $G = 0$. Then the singular values of the solution $P$ to the DLE \cref{eq:DLE_int} satisfy
 \begin{equation*}
   \sigma_k(P(t)) \le Mt^{1-2\alpha} \exp{-\eta \sqrt{k - 2\dim{Y}} },
 \end{equation*}
 for $k \ge 4 \dim{Y}$,  where $M$ and $\eta$ are positive constants independent of $t$ but dependent on $\alpha$.
\end{theorem}
After our preliminary work, the proof follows almost exactly as in~\cite{Opm15}:
\begin{proof}
  We have 
  \begin{equation*}
    \iprod{P(t)x}{y} = \int_{0}^{t}{F(z) \dif{z}}.
  \end{equation*}
  Now define $n$, $z_k$ and $w_k$ as in \cref{thm:int_approx} and define the approximation $P_m$ by
  \begin{equation*}
    P_m = h\sum_{k=-m}^{n}{w_k \exp{z_kA^*}C^*C \exp{z_kA} }.
   \end{equation*}
Since $P(t)$ and $P_m(t)$ are both self-adjoint operators and $\domain{A}$ is dense in $H$, by \cref{thm:int_approx} we then get
\begin{align*}
  \nrm{P(t) - P_m(t)} &= \sup_{\substack{z \in \domain{A}\\ \nrm{z} = 1}} \big| \iprod{\big(P(t) - P_m(t)\big)z}{z} \big| \\
                      &\le M t^{1-2\alpha}  \exp{-(2 \pi (1-2\alpha) d m)^{1/2}}.
\end{align*}
Now let
\begin{equation*}
  k_m = (2m + 2) \dim{Y}.
\end{equation*}
Since $n \le m + 1$, the rank of $P_m(t)$ is at most $k_m$, and we immediately see that we have the bound $\sigma_{k_m+1}(P(t))  \le M t^{1-2\alpha} \exp{-\eta \sqrt{m}}$ with $\eta = \sqrt{(2 \pi (1-2\alpha) d}$.
As the singular values are decreasing, we may rewrite this\footnote{
Let $k = a + bm$ with $b > 0$. 
For $j = k+1, \ldots, k+1+b$ we have
$\sigma_j \le \sigma_{k+1} \le M\exp{-\eta\sqrt{m}} \le M\exp{-\tilde{\eta}\sqrt{k-a}} \le M\exp{-\tilde{\eta}\sqrt{j-a}} \exp{-\tilde{\eta}\big(\sqrt{k-a} - \sqrt{j-a} \big)}$,
with $\tilde{\eta} = \eta / \sqrt{b}$.
Now, $\sqrt{k-a} - \sqrt{j-a} \ge \sqrt{k-a} - \sqrt{k+1+b-a} = \sqrt{bm} - \sqrt{b(m+1)+1}$. The latter function is decreasing with $m$, so we get $\sigma_j \le M\exp{\eta\big(\sqrt{2 + 1/b} - 1\big)} \exp{-\tilde{\eta}  \sqrt{j-a}}$.
}
 as
\begin{equation*}
  \sigma_{j} \le \tilde{M} t^{1-2\alpha} \exp{- \tilde{\eta}  \sqrt{j  - 2\dim{Y}}},
\end{equation*}
for $j \ge 4\dim{Y}$, with the modified constants $\tilde{M} = M\exp{-\eta \big(\sqrt{2 + 1/(2\dim Y)} - 1\big)}$  and $\tilde{\eta} = \frac{\eta}{\sqrt{2\dim{Y}}}$.
\end{proof}

\begin{remark}\label{remark:shift}
The theorem is stated for $k \ge 4 \dim Y$ since this is the maximal rank of the approximant $P_1(t)$, which provides the first explicit information we have. As the singular values are decreasing, it is of course possible to scale the constant $M$ by $\sigma_1 / \sigma_{(4\dim Y)}$ and show exponential square-root decay for $k \ge 1$. However, the bound is then also that much worse in the given interval.
\end{remark}

\begin{remark}
  In the current approach, the factor $t^{1-2\alpha}$ is desired when $t$ is small, but also means that the bound deteriorates when $t \to \infty$. This holds also in the exponentially stable case, i.e. when $\omega < 0$, because we can not bound $\exp{2\omega \Re(z)}$ \emph{uniformly} on $(0,t/2)$ by $\exp{-Mt}$ for any positive $M$.
However, when $\omega < 0$ the solution to the DLE tends to the solution of the corresponding algebraic Lyapunov equation (ALE) $0 = A^*P + PA + C^*C$ as $t \to \infty$, see e.g.~\cite[Section 2.3]{LasT00} (also for the more general Riccati case). If $\omega < 0$ and $t \in [0, T]$ where $T$ is very large the bound in \cref{thm:DLE_decay_G0} is therefore overly pessimistic, and we might instead start from the ALE decay results and consider the small perturbation arising from the difference between the ALE and DLE solutions. The ALE case was considered in~\cite{Opm15}, which uses the sinc quadrature theory for the infinite interval $(0, \infty)$~\cite[Theorem 3.9]{LunB92} applied to our function $F(z)$. (See also~\cite[Example 4.2.10]{Ste93}). The new integration interval leads to a different choice of transformation $\phi_t$, for which it is straightforward to gainfully utilize the $\exp{2\omega \Re(z)}$ term. It results in the exponential square-root decay
  \begin{equation*}
    \bigg| \int_{0}^{\infty}{F(z) \dif{z}}  - h\sum_{k=-m}^{n}{F(\exp{kh}) \exp{kh}} \bigg| \le M\exp{-\sqrt{2\pi\delta\alpha m}}.
  \end{equation*}
  By \cref{eq:F_bound} we have
  \begin{equation*}
    \bigg| \int_{0}^{T}{F(z) \dif{z}} - \int_{0}^{\infty}{F(z) \dif{z}}\bigg| \le \frac{M T^{-2\alpha} \exp{-2\omega T}}{2\omega},
  \end{equation*}
  and we thus get exponential square-root decay except for a small constant term, if $T$ is large. We note, however, that if $T$ is large it might be more worthwhile to consider the ALE with $T = \infty$ directly, rather than the DLE.
\end{remark}

\begin{remark}
Similar results are expected to hold in the nonautonomous case, i.e. when $A$, $B$ and $C$ may depend on $t$. If the operators $A(t)$ all generate analytic semigroups with the same domain $\domain{A(t)} = D$ and the map $t \mapsto A(t): [0, T] \to \cL(D, H)$ is sufficiently nice (H\"{o}lder continuous, with $D$ having the graph norm) then there is a evolution system $U(t, s)$ satisfying 
$\ddt U(t,s) = A(t) U(t, s)$ and $\nrm{(\omega I - A(t))^{\alpha}U(t,s)x} \le \frac{M}{(t-s)^{\alpha}}$ for $x \in D$. See e.g.~\cite[Section 5.6]{Paz83}. It can then be verified by differentiation that the function
\begin{equation*}
  P(t) = \int_{0}^{t}{U(t, s)^* C(s)^* C(s) U(t, s) \dif s}
\end{equation*}
solves the DLE $\dot{P}(t) = A(t)^* P(t) + P(t) A(t) + C(t)^*C(t)$, $P(0) = 0$. We can thus follow the same program as in the autonomous case if we can guarantee that $C(s) (\omega I - A(t))^{-\alpha} \in \cL(H, Y)$ with $\alpha < 1/2$ for $s$ near $t$, since then $\nrm{C(s) U(t, s) x} \le \frac{M}{(t-s)^{\alpha}}$. 
A simple example of when such a condition would hold is when the time dependency is of the form $A(t) = \kappa(t) \tilde{A}$, $C(t) = \lambda(t) \tilde{C}$, where $\tilde{A}$ and $\tilde{C}$ are fixed operators and the functions $\kappa, \lambda$ are continuous and bounded away from zero. Then it is clear that if $\tilde{A}$ and $\tilde{C}$ satisfies the assumptions for the autonomous case, also the above condition is fulfilled.
\end{remark}

A non-zero operator $G$ makes the situation more delicate. If $G$ is a finite-rank operator, then the above result is essentially just shifted by $\rank G$. For consistency, we formulate this in terms of the output space $Z$:
\begin{theorem} \label{thm:DLE_decay_finite_rank}
 Let \cref{assumption:A,assumption:C,assumption:G} be satisfied, with the output spaces $Y$ and $Z$ both having finite nonzero dimension. Then the singular values of the solution $P$ to the DLE \cref{eq:DLE_int} satisfy
 \begin{equation*}
   \sigma_k(P(t)) \le Mt^{1-2\alpha} \exp{-\eta \sqrt{k - 2\dim{Y} - \dim{Z}} },
 \end{equation*}
 for $k \ge \max(1, 4\dim{Y} + \dim{Z})$, where $M$ and $\eta$ are positive constants independent of $t$ but dependent on $\alpha$.
\end{theorem}
\begin{proof}
  This follows by the same procedure as in the proof of \cref{thm:DLE_decay_G0} after changing the definition of $P_m$ to 
  \begin{equation*}
    P_m = \exp{tA^*}G^*G \exp{tA} + h\sum_{k=-m}^{n}{w_k \exp{z_kA^*}C^*C \exp{z_kA} }.
  \end{equation*}
  In this case, $k_m = \dim{Z} + (2m + 2)\dim{Y}$.
\end{proof}

  As an alternative proof, we may make use of the well-known Weyl's inequality (also known as the Ky Fan inequality): Let $F_1$ and $F_2$ be two compact operators on $H$ with singular values $\{\sigma^1_k\}_{k=1}^{\infty}$ and $\{\sigma^2_k\}_{k=1}^{\infty}$, respectively. Denote the singular values of $F_1 + F_2$ by $\{\sigma_k\}_{k=1}^{\infty}$. Then $\sigma_{j+k-1} \le \sigma^1_j + \sigma^2_k$ for all positive integers $j$ and $k$ \cite{Fan51}.
 If $\dim{Z} < \infty$, then $G$ and $\exp{tA^*}G^*G \exp{tA}$ are both compact operators whose singular values are zero except for the first $\dim{Z}$ ones. The operator $\int_{0}^{t}{\exp{sA^*}C^*C \exp{sA} \ds }$ is also compact, since it is the limit of a sequence of finite-rank operators (see the first part of the proof for \cref{thm:DLE_decay_G0}). Hence Weyl's inequality applies, which shifts the start of the exponential decay by $\dim{Z}$.

Finally, we consider the case where $G$ is a general operator. To handle the term $\exp{tA^*}G^*G \exp{tA}$ we then have to impose stricter requirements on the semigroup $\exp{tA}$ and, by extension, its generator $A$. Alternatively, we may require that the singular values of $G$ decay sufficiently fast.

\begin{theorem} \label{thm:DLE_decay_Gfull}
 Let \cref{assumption:A,assumption:C,assumption:G} be satisfied, with the output space $Y$ having finite dimension $\dim Y \ge 1$ and $\dim Z = \infty$. If the singular values of the solution operator $\exp{tA}$ decay exponentially in the square root, $\sigma_k(\exp{tA}) \le \tilde{M} \exp{-\tilde{\eta}(t) \sqrt{k}}$, then the singular values of the solution $P$ to the DLE \cref{eq:DLE_int} satisfy
 \begin{equation*}
   \sigma_k(P(t)) \le M\max(1, t^{1-2\alpha}) \exp{-\frac{1}{2}\min(\eta, 2\tilde{\eta}(t)) \sqrt{k + 1 - 2\dim{Y} } },
 \end{equation*}
 for $k \ge 6\dim{Y} - 1$,  where $M$ and $\eta$ are positive constants independent of $t$ but dependent on $\alpha$. If instead $\sigma_k(G) \le \tilde{M}\exp{-\tilde{\eta} \sqrt{k}}$, then the same bound holds but without the time dependence in the exponent.
\end{theorem}

\begin{proof}
The extra assumption on $\exp{tA}$ in particular implies that $\exp{tA}$ is compact, and since $G$ is a bounded also $\exp{tA^*}G^*G \exp{tA}$ is compact. Further, the singular values clearly satisfy $\sigma_k(\exp{tA^*}G^*G \exp{tA}) \le \hat{M}\exp{-2\tilde{\eta} \sqrt{k}}$ for some constant $\hat{M}$. We may therefore apply Weyl's inequality as in the paragraph after the proof of \cref{thm:DLE_decay_finite_rank}. By \cref{thm:DLE_decay_G0} this directly yields
\begin{align*}
  \sigma_{2k-2\dim{Y} - 1}(P(t)) &= \sigma_{k + (k - 2\dim{Y}) - 1}(P(t))\\
 &\le Mt^{1-2\alpha} \exp{-\eta \sqrt{k - 2\dim{Y}}} + \hat{M} \exp{-2\tilde{\eta}(t) \sqrt{k - 2\dim{Y}}}\\
                      &\le 2\max(Mt^{1-2\alpha}, \hat{M})\exp{-\min(\eta, 2\tilde{\eta}(t)) \sqrt{k - 2\dim{Y}}},
\end{align*}
and thus
\begin{equation*}
  \sigma_{j}(P(t)) \le 2\max(Mt^{1-2\alpha}, \hat{M})\exp{-\frac{1}{2}\min(\eta, 2\tilde{\eta}(t)) \sqrt{j + 1 - 2\dim{Y}}},
\end{equation*}
for all $j \ge 6\dim{Y}-1$. For the second case, we note that the assumption implies that 
\begin{equation*}
  \sigma_k(\exp{tA^*}G^*G \exp{tA}) \le \hat{M}\exp{-2\tilde{\eta} \sqrt{k}},
\end{equation*}
with a different constant $\hat{M}$, due to the exponential boundedness of $\exp{tA}$. We may thus apply Weyl's inequality in exactly the same way.
\end{proof}

\begin{remark}
  When $A$ is diagonalizable, the assumption on $\exp{tA}$ obviously means that the eigenvalues of $A$ should go to $-\infty$ like the negative square root. This assumption is satisfied in many concrete applications. As an example, the Laplacian on $\Omega \subset \R^d$  with Dirichlet or Neumann boundary conditions has eigenvalues $\lambda_k(A)$ that decrease as $\lambda_k(A) = \cO(-k^{2/d})$ by Weyl's law, see e.g.~\cite[Chapter VI]{CouH53}. Hence the assumption is satisfied for such problems of up to dimension $4$.
\end{remark}

\section{Riccati equations} \label{sec:DRE}

As in~\cite{Opm15}, we may extend the Lyapunov results to the Riccati case by using a factorization into output and input-output maps. For this, we will employ the framework of well-posed systems advocated by Salamon~\cite{Sal87} and Staffans~\cite{Sta05}, see also~\cite{Mik02, TucW09}. 
As in \cref{sec:DLE} we first consider the case of a zero initial condition, then extend this to the finite-rank case and finally to the case of a general $G$ but with extra requirements on $A$.

\begin{theorem} \label{thm:DRE_decay}
  Let \cref{assumption:A,assumption:B,assumption:C,assumption:G} be satisfied, with the output spaces $Y$ and $Z$ having finite nonzero dimension. Then if $G=0$, the singular values of the solution $P$ to the DRE \cref{eq:DRE_int} satisfy
 \begin{equation*}
   \sigma_k(P(t)) \le M t^{1-2\alpha} \exp{-\eta \sqrt{k - 2\dim{Y}} },
 \end{equation*}
 for $k \ge 4\dim{Y}$.
If $G \neq 0$ but $\dim{Z} < \infty$ we instead get
\begin{equation*}
   \sigma_k(P(t)) \le M t^{1-2\alpha} \exp{-\eta \sqrt{k - 2\dim{Y} - \dim{Z}} },
 \end{equation*}
 for $k \ge 4\dim{Y} + \dim{Z}$.
If $\dim{Z} = \infty$ and $\sigma_k(\exp{tA}) \le \tilde{M} \exp{-\tilde{\eta}(t) \sqrt{k}}$, then
 \begin{equation*}
   \sigma_{k(P(t))} \le M \max(1, t^{1-2\alpha}) \exp{-\frac{1}{2}\min(\eta, 2\tilde{\eta}(t)) \sqrt{k + 1 - 2\dim{Y} } },
 \end{equation*}
 for $k \ge 6\dim{Y} - 1$. Finally, if $\dim{Z} = \infty$ and $\sigma_k(G) \le \tilde{M}\exp{-\tilde{\eta} \sqrt{k}}$, then the last bound still holds, but without the time dependency in the exponent.
 In all the cases above, $M$ and $\eta$ are positive constants independent of $t$ but dependent on $\alpha$.
\end{theorem}

\begin{remark}
As in \cref{remark:shift}, we can shift the decay to start at $k = 1$ by increasing the constant $M$, at the expense of a worse bound in the interval given above.
\end{remark}

\begin{proof}
Let the output and input-output mappings $\cC_t$ and $\cD_t$ be given by
\begin{equation*}
  (\cC_t x_0)(s) = C\exp{s A} x_0 \quad \text{and} \quad (\cD_t u)(s) = \int_{0}^{s}{C\exp{(s-\tau)A}Bu(\tau)\dif{\tau}}.
\end{equation*}
By \cite[Theorem 5.7.3]{Sta05}, these mappings satisfy $\cC_t \in \cL(H, L^2([0,t], Y))$ and $\cD_t \in \cL(L^2([0,t], U), L^2([0,t], Y))$, due to \cref{assumption:B,assumption:C}.
When $G = 0$ we can then directly apply the result of Salamon~\cite[Theorem 5.1]{Sal87}, which (in our notation) states that 
\begin{equation*}
  P(t) = \cC_t^* \big( \cI + \cD_t\cD_t^* \big)^ {-1} \cC_t.
\end{equation*}
Here, $\cI$ denotes the identity operator on $L^2([0,t], Y)$, and the inverse of $\cI + \cD_t\cD_t^*$ exists as a bounded self-adjoint operator by the Lax-Milgram lemma.
A straightforward calculation shows that $\cC_t^*$ is given by $\cC_t^*u = \int_{0}^{t}{\exp{sA^*} C^* u(s) \ds}$, and we get
\begin{equation*}
  \cC_t^*\cC_t x_0 = \int_{0}^{t}{\exp{sA^*} C^* C \exp{sA} x_0 \ds}.
\end{equation*}
Thus, in fact, for $x, y \in \domain{A}$ we have $\iprod{\cC_t^*\cC_tx}{y} = F(t)$ with $F$ defined by \cref{eq:F_def}. Hence the singular values of $\cC_t^*\cC_t$ decay exponentially in the square root, by exactly the same reasoning as in the proof of \cref{thm:DLE_decay_G0}. Multiplying $\cC_t^*\cC_t$ by the bounded operator $(\cI + \cD\cD^*)^{-1}$ only scales the singular values by the factor $\nrm{(\cI + \cD\cD^*)^{-1}}$, so we have thus proven the first assertion.

The argument in~\cite[Theorem 5.1]{Sal87} may be extended also to the more general case that $G \neq 0$. We instead get
\begin{align*}
  P(t) &=  \cC_{G,t}^*\cC_{G,t} + \cC_t^*\cC_t  \\
       &- \big( \cC_{G,t}^*\cD_{G,t} + \cC_t^*\cD_t \big) \big( \cI + \cD_t^*\cD_t + \cD_{G,t}^*\cD_{G,t}\big)^ {-1} \big( \cD_{G,t}^*\cC_{G,t} + \cD_t^*\cC_t \big),
\end{align*}
where 
\begin{equation*}
  \cC_{G,t} x_0 = G \exp{tA} x_0 \quad \text{and} \quad \cD_{G,t} u = G\lim_{s \to t} \int_{0}^{s}{\exp{(s-\tau)A}Bu(\tau)\dif{\tau}}
\end{equation*}
are the ``final-state'' versions of the $\cC$ and $\cD$ operators. By \cite[Theorem 5.7.3]{Sta05} and  \cite[Theorem A.3.7(ii)]{Sta05}, the input-output operator $t \mapsto \int_{0}^{t}{G\exp{(t-s)A}Bu(s)\ds}$ maps $u \in L^2([0,t], U)$ into $C([0, t], Z)$ under \cref{assumption:G}, and $\cD_{G,t}$ is therefore well-defined.

 Recall that the problem is stated on $t \in [0, T]$. For any such $t$, we define the product space $X_t = L^2([0,t], Y) \times Z$ with the induced topology
\begin{equation*}
  \nrm{
    \begin{bmatrix}
      y\\
           z
    \end{bmatrix}
  }_{X_t}^2
  = \nrm{y}_{L^2([0,t], Y)}^2 + \nrm{z}_Z^2.
\end{equation*}
Further let the operators $\tcC_t : H \to X_t$ and $\tcD_t : L^2([0,t], U) \to X_t$ be defined by
\begin{equation*}
  \tcC_t =
  \begin{bmatrix}
    \cC_t \\
    \cC_{G,t} 
  \end{bmatrix}
  \quad \text{and} \quad
  \tcD_t =
  \begin{bmatrix}
    \cD_t \\
    \cD_{G,t} 
  \end{bmatrix}.
\end{equation*}
Then clearly $\tcC_t$ and $\tcD_t$ are linear and bounded with adjoints $\tcC_t^* : X_t \to H$ and $\tcD_t^* : X_t \to L^2([0,t], U)$ given by 
\begin{equation*}
  \tcC_t^* =
  \begin{bmatrix}
    \cC_t^* & \cC_{G,t}^* 
  \end{bmatrix}
  \quad \text{and} \quad 
  \tcD_t^* =
  \begin{bmatrix}
    \cD_t^* & \cD_{G,t}^* 
  \end{bmatrix}
  .
\end{equation*}
 It follows that we can factorize the above expression for $P(t)$ as
\begin{equation*}
  P(t) = \tcC_t^* \big( \cI + \tcD_t\tcD_t^* \big)^ {-1} \tcC_t.
\end{equation*}
Hence, the singular value decay of $P(t)$ is the same as that of $\tcC_t^* \tcC_t \in \cL(H)$, i.e.\ of $\cC_t^*\cC_t + \cC_{G,t}^*\cC_{G,t} = \cC_t^*\cC_t + \exp{tA^*} G^*G \exp{tA}$. 
Applying Weyl's inequality with either the assumption that $\dim{Z} < \infty$ or that the singular values of $\exp{tA}$ or $G$ decay sufficiently fast yields the second, third and fourth assertions, as in the proofs of \cref{thm:DLE_decay_finite_rank,thm:DLE_decay_Gfull}.
\end{proof}

\begin{remark}\label{remark:weighting}
  The above theorem extends to the case of a more general cost functional with a coercive weighting term $
  \begin{bmatrix}
    Q & N \\
    N^* & R
  \end{bmatrix}  $ in much the same way as~\cite{Opm15}. Since $N=0$ in most practical applications and $Q$ and $R$ may be included in $C$ and $B$, respectively, we choose to omit this from the theorem and proof in order to simplify the notation.
\end{remark}

We note that while we have only shown that the given assumptions are sufficient for fast decay of the singular values, we do not claim that they are necessary conditions. Nevertheless, violating one of the assumptions generally either leads to a not well-defined problem or slow decay. See e.g.~\cite{Opm15} for a number of examples in the algebraic setting. As an additional example, consider the advection equation $\ddt x(t, \xi) = \frac{\dif}{\dif{\xi}} x(t, \xi)$ on $\xi \in (0, \infty)$, with $x(0, \xi) = x_0(\xi)$. The solution is given by $x(t, \xi) = x_0(t)$, i.e.\ it simply shifts the initial condition to the left. If the output operator $C$ is the trace of $x$ at $0$, the output map is given by $\cC_tx_0 = x_0(\cdot)$. This means that
\begin{equation*}
  \nrm{\cC_tx_0}_{L^2([0,t], Y)}^2 = \nrm{x_0}_{L^2([0,t], H)}^2,
\end{equation*}
and $\cC_t$ is therefore a partial isometry for any $t > 0$. Since $H$ is infinite-dimensional, $\cC_t$ has infinitely many singular values that are equal to $1$. The solution to the corresponding differential Lyapunov equation therefore exhibits no decay of its singular values at all. The main problem here is the lack of analyticity of the operator $\frac{\dif}{\dif{\xi}}$. (Cf.~\cite[Section 8.7A]{LasT00b}.)

On the other hand, analyticity is sometimes not strictly necessary when $B$ and $C$ are bounded operators. This is demonstrated for the algebraic case in~\cite{CurS01}, which shows that the solution is nuclear. That means that $\sum_{k = 1}^{\infty}{\sigma_k(P)} < \infty$, i.e.\ the singular values decay to zero at least as fast as $1/k$, but not necessarily as fast as $\exp{- \gamma \sqrt{k}}$. (These results extend to the differential case.)

\section{Numerical experiments} \label{sec:experiments}
To demonstrate the applicability of the bounds proposed in \cref{thm:DLE_decay_G0,thm:DLE_decay_finite_rank,thm:DLE_decay_Gfull,thm:DRE_decay} we have performed a few numerical experiments. 
In all cases, we consider DRE/DLEs arising from LQR problems with the state and output equations given by
\begin{align}
  \dot{x} &= A x + B u, \quad x(0) = x_0, \label{eq:state}\\
  y &= C x, \label{eq:output}
\end{align}
The solution $P$ to the DRE associated with the operators $A$, $B$ and $C$ yields the optimal input function $u^{\text{opt}}$ in feedback form; $u^{\text{opt}}(t) = - B^* P(T-t) x(t)$. It is optimal in the sense that it minimizes the cost functional
\begin{equation*}
  J(u) = \int_{0}^{T} { \nrm{y}_Y^2 + \nrm{u}_U^2 \,\dt} + \nrm{G x(T)}_Z^2.
\end{equation*}
The aim is thus to drive the output $y$ to zero while being mindful of the cost $\nrm{u}^2$ of doing so. In the extended case mentioned in \cref{remark:weighting}, the weighting factors scale the relative costs of $y$ and $u$, respectively. When $B = 0$, the solution to the corresponding DLE yields the observability Gramian, an indicator of which states $x$ that can be detected by using only the output $y$.

In all the following examples we consider the domain $\Omega = [0, 1]^2$ to be the unit square, with boundary $\Gamma$. We further let the state space be $H = L^2(\Omega)$ except where otherwise noted. We choose $A = \Delta : \domain{A} \subset H \to H$ to be the Laplacian. Since we will vary the boundary conditions, its domain will change as well. We can, however, always consider it to be generated by the inner product $a(u,v) = \int_{\Omega}{ \nabla u \cdot \nabla v}$, where $u, v \in V = \domain{(-A)^{1/2}}$. In the case of homogeneous Dirichlet boundary conditions, we have $\domain{A} = H^2 \cap H^1_0(\Omega)$ and $V = H^1_0(\Omega)$. We note that \cref{assumption:A} is satisfied, with the region of analyticity being the entire right halfplane.

Since we cannot investigate the infinite-dimensional case in finite precision arithmetic, a discretization of the equation is required. For the spatial discretization, we have used the finite element method based on the inner product $a$. For a given mesh size $h$, we get the finite element space $V_h \subset V \subset H$ and the approximate solution $P_h$ is an operator from $V_h$ to $V_h$. We may, however, extend it to an operator on $H$ by forming $\mathcal{I}_h P_h \mathcal{P}_h$ where $\mathcal{I}_h : V_h \to H$ denotes the identity operator and $\mathcal{P}_h : H \to V_h$ is the $a$-orthogonal projection onto the finite element space. For a detailed account of the resulting matrix-valued equations, see e.g.~\cite[Section 5]{MalPS18}. We generate the respective matrices here by using the library FreeFem++~\cite{Hec12}, with $P2$ conforming finite elements unless otherwise noted.

 Further, since the discretized DLE/DREs are matrix-valued and their solutions are typically dense, it is not feasible to simply transform these into vector-valued ODEs and solve them directly. We use instead the MATLAB package DREsplit\footnote{Available from the author via email on request, or from \url{www.tonystillfjord.net}.} developed by the author to compute accurate low-rank approximations to the solutions. The reported singular values are thus not exact, but the integration parameters were chosen in such a way that further refining the temporal discretizations has a negligible effect on the end results. In particular, we used the second-order Strang splitting with $256$ time steps. This requires the computation of many matrix exponential actions, and for this a basic block Krylov subspace method with residual norm tolerance $10^{-4}$ was employed. The relative tolerance for the low-rank approximation was set to the round-off error level. For further details on the use of splitting schemes in this context, see e.g.~\cite{Sti17} or~\cite{Sti15}.

With this said, we want to note that the reported results also provide some insight into how the discretized equations converge to their infinite-dimensional counterparts.

\subsection{Example 1}\label{example:bounded}
We consider first the bounded Lyapunov case by taking the input operator $B = 0$ and letting the output be the mean of the solution. More specifically, we take $Y = \R$ and set $C : H \to Y$, $C x = \int_{\Omega}{x}$. Then clearly $\nrm{Cx}_{\R} \le \nrm{x}_H$, since $\Omega$ is the unit square. We thus have $\beta = 0$ and $\alpha = 0$. Further setting $G = 0$  implies that \cref{assumption:B,assumption:C,assumption:G} are satisfied. To complete the specification of $A$, we choose homogeneous Dirichlet boundary conditions.

We computed the singular values for a number of different spatial discretizations, starting with a grid that has $N = 9$ internal nodes and refining this 6 times. Each refinement roughly halves the mesh size and thus roughly quadruples the number of nodes, leading to meshes with $N = 9, 49, 225, 961, 3969, 16129$ and $65025$ internal nodes, respectively.
\Cref{fig:ex01} shows the computed singular values of the solutions (the $\LHH$-extended operators, not the matrices) for different spatial discretizations, at the final time $T = 0.1$. The curves are ordered in size from bottom to top, i.e.\ the lowermost curve corresponds to the $N = 9$ discretization, while the topmost corresponds to the $N = 65025$ discretization. We observe that while the initial decay is very much exponential in nature, when we refine the discretization the decay worsens and tends to the exponential square root bound. This is precisely the same behaviour as as seen in the algebraic case in e.g.~\cite{GruK14}.

\begin{figure}[htbp]
  \centering
%
%
\begin{tikzpicture}

\begin{axis}[%
width=0.761\columnwidth,
height=0.494\columnwidth,
at={(0\columnwidth,0\columnwidth)},
scale only axis,
xmin=0,
xmax=30,
xlabel style={font=\color{white!15!black}},
xlabel={$k$},
ymode=log,
ymin=1e-25,
ymax=1,
yminorticks=true,
ylabel style={font=\color{white!15!black}},
ylabel={$\sigma_k$},
axis background/.style={fill=white},
title style={font=\bfseries},
title={Singular values for different discretizations},
legend style={legend cell align=left, align=left, draw=white!15!black},
xtick distance= 5,
legend pos = north east
]
\addplot [color=black, mark=+, mark options={solid, black}, forget plot]
  table[row sep=crcr]{%
1	0.0161854613032337\\
2	0.000215390188225665\\
3	1.32515579072567e-06\\
4	5.43169372332684e-08\\
5	1.07169988640604e-17\\
6	0\\
};
\addplot [color=black, mark=+, mark options={solid, black}, forget plot]
  table[row sep=crcr]{%
1	0.0167629970555036\\
2	0.000391113456142698\\
3	1.80567540993429e-05\\
4	7.18211630349624e-07\\
5	3.33881758126985e-08\\
6	2.86093572645083e-09\\
7	5.75936502309276e-10\\
8	3.57342519212328e-11\\
9	7.44373512963773e-13\\
10	3.87249662788405e-14\\
11	2.30093503179367e-15\\
12	2.61230526058018e-16\\
13	1.18042498424209e-17\\
14	8.2421007803113e-18\\
15	7.31523001834231e-18\\
16	3.40017284017379e-18\\
17	0\\
18	0\\
};
\addplot [color=black, mark=+, mark options={solid, black}, forget plot]
  table[row sep=crcr]{%
1	0.0167947261921203\\
2	0.000417140118040125\\
3	2.97009347512814e-05\\
4	2.71908476732903e-06\\
5	3.08560354360297e-07\\
6	2.69416593807966e-08\\
7	3.76015781019299e-09\\
8	6.4216974522819e-10\\
9	6.81129301646907e-11\\
10	8.52069956666957e-12\\
11	1.647806231086e-12\\
12	4.7415131633166e-13\\
13	6.11036825424766e-14\\
14	9.95139138724096e-15\\
15	6.55618050315033e-16\\
16	4.48557109970522e-17\\
17	9.84345859021824e-18\\
18	5.17116889634047e-18\\
19	2.44770787113503e-18\\
20	0\\
21	0\\
22	0\\
};
\addplot [color=black, mark=+, mark options={solid, black}, forget plot]
  table[row sep=crcr]{%
1	0.0167961391012838\\
2	0.000417300585796808\\
3	3.2359883083583e-05\\
4	3.99177800517742e-06\\
5	5.64351011179537e-07\\
6	9.9595492113138e-08\\
7	1.66012394334302e-08\\
8	3.34637919593602e-09\\
9	8.49485624526092e-10\\
10	1.42259662637532e-10\\
11	2.53183365130746e-11\\
12	4.66293370991613e-12\\
13	1.31156191511982e-12\\
14	2.61708875567459e-13\\
15	4.0292334266029e-14\\
16	5.64242611347229e-15\\
17	1.7465395948213e-15\\
18	5.62852206594032e-16\\
19	1.61061141253904e-16\\
20	1.76817714273866e-17\\
21	5.79478234219275e-18\\
22	3.32268558594502e-18\\
23	3.06664647831241e-18\\
24	0\\
25	0\\
26	0\\
27	0\\
};
\addplot [color=black, mark=+, mark options={solid, black}, forget plot]
  table[row sep=crcr]{%
1	0.0167962276146634\\
2	0.000417162753025221\\
3	3.23553363887149e-05\\
4	4.26331149476575e-06\\
5	7.45910359816752e-07\\
6	1.45445208871416e-07\\
7	3.22299710399686e-08\\
8	8.53725123264275e-09\\
9	2.40223616073021e-09\\
10	6.79373040787907e-10\\
11	1.76627688802216e-10\\
12	3.77656502303638e-11\\
13	8.32367318099346e-12\\
14	2.41636714188436e-12\\
15	7.52993516433568e-13\\
16	2.313256274125e-13\\
17	5.17318948236514e-14\\
18	1.10092281865893e-14\\
19	2.36193197522551e-15\\
20	6.47352022206208e-16\\
21	1.75075923442364e-16\\
22	3.53545321245855e-17\\
23	1.70459073769892e-17\\
24	1.08259445588994e-17\\
25	5.36672092410053e-18\\
26	5.22627300550371e-18\\
27	3.94672097717131e-18\\
28	3.25909468561849e-18\\
29	2.63116645670249e-18\\
30	0\\
31	0\\
32	0\\
33	0\\
34	0\\
35	0\\
};
\addplot [color=black, mark=+, mark options={solid, black}, forget plot]
  table[row sep=crcr]{%
1	0.0167962330900944\\
2	0.000417149476267125\\
3	3.23346984865921e-05\\
4	4.25924432779738e-06\\
5	7.70863849643722e-07\\
6	1.71577685974585e-07\\
7	4.28995409674669e-08\\
8	1.17375677884358e-08\\
9	3.80932944821284e-09\\
10	1.41637903546781e-09\\
11	4.75592371120149e-10\\
12	1.34412929619433e-10\\
13	3.8505512806228e-11\\
14	1.12479268143287e-11\\
15	3.25266059759794e-12\\
16	1.10421128849998e-12\\
17	3.70408555567187e-13\\
18	9.95322200688045e-14\\
19	2.3479416366295e-14\\
20	6.09287945727971e-15\\
21	1.64356793527049e-15\\
22	4.76365680495926e-16\\
23	1.19832222797503e-16\\
24	2.64394304853216e-17\\
25	5.03143641538427e-18\\
};
\addplot [color=black, mark=+, mark options={solid, black}]
  table[row sep=crcr]{%
1	0.0167962353514286\\
2	0.000417148558625769\\
3	3.23327240043817e-05\\
4	4.25682274459516e-06\\
5	7.69702782477717e-07\\
6	1.73217970537683e-07\\
7	4.60137174534958e-08\\
8	1.39276137589277e-08\\
9	4.73029588072187e-09\\
10	1.77729279628029e-09\\
11	6.64186999130944e-10\\
12	2.32642550975019e-10\\
13	7.56701663880917e-11\\
14	2.24149546481149e-11\\
15	6.42218391676439e-12\\
16	2.02470095855338e-12\\
17	6.76334814771927e-13\\
18	1.90268160521678e-13\\
19	4.6393966215908e-14\\
20	1.07041412912521e-14\\
21	2.59411088154619e-15\\
22	7.41036114415074e-16\\
23	1.95129293506874e-16\\
24	3.84598563983988e-17\\
25	9.69552375931618e-18\\
26	8.93867510025674e-18\\
27	5.53815561033077e-18\\
28	4.67467760346888e-18\\
29	0\\
30	0\\
31	0\\
};
\addlegendentry{$\sigma_k$}

\addplot [color=black, line width=2.0pt]
  table[row sep=crcr]{%
1	2.42323283817815e-05\\
2	3.62798271733641e-07\\
3	5.43169372332683e-09\\
4	8.13214918666667e-11\\
5	1.21751802960088e-12\\
6	1.82282705146832e-14\\
7	2.72907536379885e-16\\
8	4.08587986188511e-18\\
9	6.11724193007253e-20\\
10	9.15852890832989e-22\\
};
\addlegendentry{$C\mathrm{e}^{-\gamma_1 k}$}

\addplot [color=black, dashed, line width=2.0pt]
  table[row sep=crcr]{%
1	0.167962353514286\\
2	0.00703948153441602\\
3	0.000617186364136657\\
4	7.92901554181893e-05\\
5	1.30034593221359e-05\\
6	2.53641126793967e-06\\
7	5.64228649048748e-07\\
8	1.39276137589277e-07\\
9	3.74305825960333e-08\\
10	1.08017041122756e-08\\
11	3.31233955675843e-09\\
12	1.07060190309438e-09\\
13	3.62382716211662e-10\\
14	1.27784902419105e-10\\
15	4.67401783885863e-11\\
16	1.76698923856198e-11\\
17	6.88312243285326e-12\\
18	2.75557829180334e-12\\
19	1.13119804178698e-12\\
20	4.75240203494739e-13\\
};
\addlegendentry{$C\mathrm{e}^{-\gamma_2 \sqrt{k}}$}

\end{axis}
\end{tikzpicture}%
  \caption{The singular values of the solutions computed in \nameref{example:bounded}, at the final time $T = 0.1$. They increase monotonically, and thus the lower-most line corresponds to $N = 9$ while the top-most corresponds to $N = 65025$.}
  \label{fig:ex01}
\end{figure}

\subsection{Example 2}\label{example:DRE_semibdd}
In the second example, we change the boundary conditions of $A$ to be homogeneous Dirichlet on the left edge $\Gamma_L$ and homogeneous Neumann on the top and bottom edges $\Gamma_T$, $\Gamma_B$. On the right edge, $\Gamma_R$, we apply a nonhomogeneous Neumann boundary condition, through which we control the system. That is, we set $U = \R$ and define $B : U  \to \domain{A^*}'$ by $B u = -(AN \mathbbm{1}) u$, where the function $\mathbbm{1} \in L^2(\Gamma_R)$ is constant equal to $1$ everywhere and $N : L^2(\Gamma_R) \to H^{3/2}(\Omega)$ denotes the Neumann operator implicitly defined by $Nv = w$ if $Aw = 0$ in $\Omega$, $\frac{\partial w}{\partial \nu}\rvert_{\Gamma_R} = v$, $w_{\rvert_{\Gamma_L}} = 0$ and $\frac{\partial w}{\partial \nu}\rvert_{\Gamma_T \cap \Gamma_B} = 0$. For further details on this construction, see e.g.~\cite[Section 3]{LasT00}. That $N$ maps into $H^{3/2}(\Omega)$ follows by~\cite[Thm. 8.3]{LioM72} and shows that $(-A)^{-\beta}B \in \cL(U, H)$ for $\beta = 1/4 + \epsilon$, $\epsilon > 0$.

We note that we could equally well take $U = L^2(\Gamma)$ in the continuous setting and let the input $u$ vary along the whole edge. However, for the numerics we would then have to discretize also this function, leading to one more layer of complexity.

As the output, we again use the mean of the solution over the whole domain $\Omega$, meaning that $\alpha = 0$. We discretize the system in the same way as in \nameref{example:bounded}, but because of the three Neumann edges we now have a slightly higher number of degrees of freedom for each level of discretization. The matrices are in this case of size $N = 20, 72, 272, 1056, 4160, 16512$ and $65792$, respectively.

\Cref{fig:ex02_1} shows the computed singular values of the solutions at the final time $T = 0.1$. The curves are again ordered in size from coarse (bottom) to fine (top) discretizations. We note that these results are quite similar to the results in \cref{fig:ex01}, i.e.\ the input operator does not make the situation worse, as predicted by \cref{thm:DRE_decay}.

\begin{figure}[htbp]
  \centering
%
%
\begin{tikzpicture}

\begin{axis}[%
width=0.761\columnwidth,
height=0.494\columnwidth,
at={(0\columnwidth,0\columnwidth)},
scale only axis,
xmin=0,
xmax=25,
xlabel style={font=\color{white!15!black}},
xlabel={$k$},
ymode=log,
ymin=1e-18,
ymax=1,
yminorticks=true,
ylabel style={font=\color{white!15!black}},
ylabel={$\sigma_k$},
axis background/.style={fill=white},
title style={font=\bfseries},
title={Singular values for different discretizations},
legend style={legend cell align=left, align=left, draw=white!15!black},
xtick distance= 5,
legend pos = north east
]
\addplot [color=black, mark=+, mark options={solid, black}, forget plot]
  table[row sep=crcr]{%
1	0.0655398398274146\\
2	0.000578607855219622\\
3	2.44975719475123e-05\\
4	1.2295130350608e-06\\
5	3.03387498654549e-08\\
6	1.01541744079293e-09\\
7	1.03229084612143e-10\\
8	5.13401924086131e-12\\
9	1.98640885425879e-13\\
10	9.15776351016758e-15\\
11	5.0295312479009e-16\\
12	2.4444222514717e-17\\
13	2.05649506143923e-17\\
14	0\\
};
\addplot [color=black, mark=+, mark options={solid, black}, forget plot]
  table[row sep=crcr]{%
1	0.065486528713183\\
2	0.000630713254078722\\
3	4.17873200595129e-05\\
4	3.81987207003632e-06\\
5	4.13829549195372e-07\\
6	3.6393218863526e-08\\
7	3.32264620644011e-09\\
8	3.41547227337853e-10\\
9	1.32674967355427e-11\\
10	6.69092416222559e-13\\
11	1.41149264556106e-13\\
12	1.31521029471748e-14\\
13	1.55400425281195e-15\\
14	1.75443171003239e-16\\
15	5.44146947959405e-17\\
16	3.58050969348619e-17\\
17	2.36817762390759e-17\\
18	1.66459590719424e-17\\
19	5.24332610402245e-18\\
};
\addplot [color=black, mark=+, mark options={solid, black}, forget plot]
  table[row sep=crcr]{%
1	0.0654820409556706\\
2	0.000629130080552123\\
3	4.74239153082931e-05\\
4	5.86719038411524e-06\\
5	8.28970163448883e-07\\
6	1.45288562408937e-07\\
7	2.03459835469647e-08\\
8	3.14177847739439e-09\\
9	4.27897389174787e-10\\
10	8.1561603431343e-11\\
11	1.21801683581693e-11\\
12	1.33018442968197e-12\\
13	1.44064025028268e-13\\
14	1.50481415491338e-14\\
15	1.45549152433834e-15\\
16	1.37777418359526e-16\\
17	5.52770338767917e-17\\
18	3.11648277976232e-17\\
19	0\\
};
\addplot [color=black, mark=+, mark options={solid, black}, forget plot]
  table[row sep=crcr]{%
1	0.0654816570641727\\
2	0.000628587774841944\\
3	4.74256588265342e-05\\
4	6.4432787790063e-06\\
5	1.14287566200219e-06\\
6	2.19553852800154e-07\\
7	4.75194067477135e-08\\
8	1.08724134043148e-08\\
9	2.08629042074346e-09\\
10	4.43509795615433e-10\\
11	8.65736762427213e-11\\
12	1.79750853601984e-11\\
13	4.27241065523942e-12\\
14	9.69470606340366e-13\\
15	1.72790806158448e-13\\
16	2.89093817827406e-14\\
17	4.53307405681391e-15\\
18	8.54501936623667e-16\\
19	1.32711035526243e-16\\
20	4.71560831142207e-17\\
21	1.95072260121149e-17\\
22	0\\
23	0\\
};
\addplot [color=black, mark=+, mark options={solid, black}, forget plot]
  table[row sep=crcr]{%
1	0.0654816336464815\\
2	0.000628539904326765\\
3	4.73745605955742e-05\\
4	6.4401145905293e-06\\
5	1.19843075838849e-06\\
6	2.68517223410753e-07\\
7	6.51052697004099e-08\\
8	1.64325116020091e-08\\
9	4.58527454630105e-09\\
10	1.28226291083363e-09\\
11	3.18599914688575e-10\\
12	8.33631297252864e-11\\
13	2.22520844936899e-11\\
14	5.27897913166358e-12\\
15	1.45353677021576e-12\\
16	4.11331699973545e-13\\
17	1.19908203088738e-13\\
18	3.19592786159679e-14\\
19	7.07595795158176e-15\\
20	1.70179996260006e-15\\
21	3.60587147333691e-16\\
22	9.52443848996126e-17\\
23	3.72277546438329e-17\\
24	3.06768628555211e-17\\
25	1.22883811916784e-17\\
};
\addplot [color=black, mark=+, mark options={solid, black}, forget plot]
  table[row sep=crcr]{%
1	0.0654816395551749\\
2	0.000628536625488465\\
3	4.73695474043616e-05\\
4	6.43482325477773e-06\\
5	1.19659080134508e-06\\
6	2.72689227175477e-07\\
7	7.15914924828911e-08\\
8	2.04805147823209e-08\\
9	6.06326630015468e-09\\
10	1.85347148950474e-09\\
11	6.09230240596622e-10\\
12	2.02582056191196e-10\\
13	6.19118391535934e-11\\
14	1.87823431047457e-11\\
15	6.07798945049801e-12\\
16	1.85626153056814e-12\\
17	5.28013941393038e-13\\
18	1.60287858382205e-13\\
19	4.85717702309535e-14\\
20	1.31992040803635e-14\\
21	3.67684668040423e-15\\
22	1.04265616192016e-15\\
23	2.58547194036586e-16\\
24	6.11257224401442e-17\\
25	1.47326974291696e-17\\
};
\addplot [color=black, mark=+, mark options={solid, black}]
  table[row sep=crcr]{%
1	0.0654816392477379\\
2	0.000628536399189709\\
3	4.73691939849141e-05\\
4	6.43432179676452e-06\\
5	1.19602052066092e-06\\
6	2.72302683834923e-07\\
7	7.17100974744284e-08\\
8	2.10960306685043e-08\\
9	6.7403914507323e-09\\
10	2.26369075843828e-09\\
11	7.74881661166761e-10\\
12	2.68389069763179e-10\\
13	9.50856573522614e-11\\
14	3.38371876819738e-11\\
15	1.14514234464002e-11\\
16	3.61004694590896e-12\\
17	1.07774746754184e-12\\
18	3.10638005944296e-13\\
19	8.74176245464464e-14\\
20	2.40692348189261e-14\\
21	6.46557471604317e-15\\
22	1.67418272094102e-15\\
23	4.15320333906591e-16\\
24	1.01564872766246e-16\\
25	2.10947401059876e-17\\
26	0\\
27	0\\
28	0\\
};
\addlegendentry{$\sigma_k$}

\addplot [color=black, line width=2.0pt]
  table[row sep=crcr]{%
1	0.000174146451546374\\
2	4.62725979588192e-06\\
3	1.2295130350608e-07\\
4	3.26694927466529e-09\\
5	8.68063799153478e-11\\
6	2.3065395145383e-12\\
7	6.12872525880549e-14\\
8	1.62846866750683e-15\\
9	4.32701759185779e-17\\
10	1.1497354302132e-18\\
};
\addlegendentry{$C\mathrm{e}^{-\gamma_1 k}$}

\addplot [color=black, dashed, line width=2.0pt]
  table[row sep=crcr]{%
1	0.654816392477379\\
2	0.0221530888200105\\
3	0.00164792636656506\\
4	0.000184319488433547\\
5	2.675472573751e-05\\
6	4.67343514668881e-06\\
7	9.39290749796162e-07\\
8	2.10960306685044e-07\\
9	5.18827479071979e-08\\
10	1.376728495922e-08\\
11	3.89790601235019e-09\\
12	1.16736797663134e-09\\
13	3.67268707349632e-10\\
14	1.20707041648479e-10\\
15	4.12528013896877e-11\\
16	1.4604096144572e-11\\
17	5.33804481237004e-12\\
18	2.00894111689058e-12\\
19	7.76580809767195e-13\\
20	3.07704025160056e-13\\
};
\addlegendentry{$C\mathrm{e}^{-\gamma_2 \sqrt{k}}$}

\end{axis}
\end{tikzpicture}%
  \caption{The singular values of the solutions computed in \nameref{example:DRE_semibdd}, at the final time $T = 0.1$. They increase monotonically, and thus the lower-most line corresponds to $N = 20$ while the top-most corresponds to $N = 65792$.}
  \label{fig:ex02_1}
\end{figure}

We have additionally plotted the largest singular value of the finest discretized problem as a function of time in \cref{fig:ex02_2}. We note that it grows roughly as $t^{1}$, corresponding well to the factor $t^{1-2\alpha}$ predicted by \cref{thm:DRE_decay}.

\begin{figure}[htbp]
  \centering
%
%
\begin{tikzpicture}

\begin{axis}[%
width=0.761\columnwidth,
height=0.494\columnwidth,
at={(0\columnwidth,0\columnwidth)},
scale only axis,
xmin=0,
xmax=0.1,
xlabel style={font=\color{white!15!black}},
xlabel={$t$},
ymode=log,
ymin=0.0001,
ymax=0.1,
yminorticks=true,
axis background/.style={fill=white},
title style={font=\bfseries},
title={Largest singular value over time},
legend style={legend cell align=left, align=left, draw=white!15!black},
xtick distance= 0.02,
legend pos = south east
]
\addplot [color=black, mark=asterisk, mark options={solid, black}]
  table[row sep=crcr]{%
0.1	0.0654816392477379\\
0.05	0.0377701174889353\\
0.025	0.0206696208825295\\
0.0125	0.0109671344690677\\
0.00625	0.00570755144622061\\
0.003125	0.00293308657093698\\
0.0015625	0.00149461587867307\\
0.00078125	0.000757241284046826\\
0.000390625	0.000382134692591297\\
};
\addlegendentry{$\sigma_1$}

\addplot [color=black, line width=2.0pt]
  table[row sep=crcr]{%
0.1	0.0978264813033721\\
0.05	0.0489132406516861\\
0.025	0.024456620325843\\
0.0125	0.0122283101629215\\
0.00625	0.00611415508146076\\
0.003125	0.00305707754073038\\
0.0015625	0.00152853877036519\\
0.00078125	0.000764269385182595\\
0.000390625	0.000382134692591297\\
};
\addlegendentry{$C t^{1-0}$}

\addplot [color=black, dashed, line width=2.0pt]
  table[row sep=crcr]{%
0.1	0.00611415508146076\\
0.05	0.00432336051932709\\
0.025	0.00305707754073038\\
0.0125	0.00216168025966354\\
0.00625	0.00152853877036519\\
0.003125	0.00108084012983177\\
0.0015625	0.000764269385182595\\
0.00078125	0.000540420064915886\\
0.000390625	0.000382134692591297\\
};
\addlegendentry{$Ct^{1-1/2}$}

\end{axis}
\end{tikzpicture}%
  \caption{The largest singular value of the solution with $N = 65792$ computed in \nameref{example:DRE_semibdd}, plotted over time.}
  \label{fig:ex02_2}
\end{figure}

\subsection{Example 3}\label{example:DRE_unbdd}
Now consider the same setting as in the previous example, but with an unbounded output as well. More precisely, we take $Y = \R$ and define $C$ as the integral of the boundary trace over $\Gamma_T \cap \Gamma_B$:
\begin{equation*}
  C x = \int_{\Gamma_T \cap \Gamma_B}{ x_{\rvert_{\Gamma}}(s) \dif{s}}.
\end{equation*}
By~\cite[Theorem 8.3]{LioM72}, the map $x \mapsto x_{\rvert_{\Gamma}}$ belongs to $\cL(H^{1/2}(\Omega), L^2(\Gamma))$ and hence the map $CA^{-\alpha}$ is bounded for $\alpha = 1/4 + \epsilon$, $\epsilon > 0$.

With the same discretizations as in \nameref{example:DRE_semibdd}, the behaviour of the singular values is similar to when $C$ was bounded. The decay is, however, noticeably slower, as shown in \cref{fig:ex03_1}. The effect of a larger $\alpha$ can also clearly be seen when plotting the singular values for a specific discretization over time. \cref{fig:ex03_2} again shows the largest singular value for the finest discretization. We note that in comparison to \cref{fig:ex02_2}, the increase is now close to $t^{1/2}$ rather than $t^1$. Since $\alpha = 1/4$, this is in good agreement with the factor $t^{1-2\alpha}$ predicted by \cref{thm:DRE_decay}.

\begin{figure}[htbp]
  \centering
%
%
\begin{tikzpicture}

\begin{axis}[%
width=0.761\columnwidth,
height=0.494\columnwidth,
at={(0\columnwidth,0\columnwidth)},
scale only axis,
xmin=0,
xmax=35,
xlabel style={font=\color{white!15!black}},
xlabel={$k$},
ymode=log,
ymin=1e-18,
ymax=100,
yminorticks=true,
ylabel style={font=\color{white!15!black}},
ylabel={$\sigma_k$},
axis background/.style={fill=white},
title style={font=\bfseries},
title={Singular values for different discretizations},
legend style={legend cell align=left, align=left, draw=white!15!black},
xtick distance= 5,
legend pos = north east
]
\addplot [color=black, mark=+, mark options={solid, black}, forget plot]
  table[row sep=crcr]{%
1	0.324626005896867\\
2	0.0662463784040784\\
3	0.00961845146910709\\
4	0.000240794310585473\\
5	3.68417402205364e-05\\
6	3.18469221146993e-06\\
7	1.63520140633699e-07\\
8	1.44232472072261e-08\\
9	5.05762405641594e-10\\
10	5.11742364820082e-12\\
11	2.25325755789147e-13\\
12	5.47205141037533e-15\\
13	3.83624483747043e-16\\
14	2.8568706408669e-16\\
15	1.62831249128187e-16\\
16	9.41371760485492e-17\\
17	0\\
};
\addplot [color=black, mark=+, mark options={solid, black}, forget plot]
  table[row sep=crcr]{%
1	0.32118569093864\\
2	0.0667881090861144\\
3	0.0168144712539209\\
4	0.00232207386874533\\
5	0.000196803483836662\\
6	3.34084881319168e-05\\
7	5.55310002214867e-06\\
8	8.45499257506543e-07\\
9	9.46298820622626e-08\\
10	1.38207481216139e-08\\
11	1.66255461266688e-09\\
12	2.10539237430392e-10\\
13	1.53265266854034e-11\\
14	1.69342901047427e-12\\
15	1.06715372091342e-13\\
16	1.13283089599284e-14\\
17	1.63075702968878e-15\\
18	8.76639252767054e-16\\
19	4.3888564270035e-16\\
20	1.03271489841085e-16\\
21	0\\
22	0\\
23	0\\
};
\addplot [color=black, mark=+, mark options={solid, black}, forget plot]
  table[row sep=crcr]{%
1	0.320311579434911\\
2	0.0626386790890421\\
3	0.0200858842854625\\
4	0.00479033356675081\\
5	0.000894560092066162\\
6	0.000133105582378534\\
7	2.76016674202601e-05\\
8	9.19482499058932e-06\\
9	1.67084414837025e-06\\
10	3.42883625465069e-07\\
11	1.03274895955836e-07\\
12	1.08251978686431e-08\\
13	3.33174475418396e-09\\
14	3.62033969019132e-10\\
15	1.12595336087794e-10\\
16	1.32262109433307e-11\\
17	2.25317727932761e-12\\
18	3.42084089336299e-13\\
19	5.42955166403647e-14\\
20	8.50397840649146e-15\\
21	1.29480427811079e-15\\
22	1.6388765416848e-16\\
};
\addplot [color=black, mark=+, mark options={solid, black}, forget plot]
  table[row sep=crcr]{%
1	0.32019071594587\\
2	0.0610700453642103\\
3	0.0190589528543378\\
4	0.00667006113702053\\
5	0.00172635716072347\\
6	0.000423646047701593\\
7	9.4058008339867e-05\\
8	2.55425497769672e-05\\
9	9.49310260164132e-06\\
10	2.29251171459633e-06\\
11	5.59762408797652e-07\\
12	2.11997205697606e-07\\
13	6.51194582384895e-08\\
14	1.34037392244842e-08\\
15	3.05615723958962e-09\\
16	1.0036264145044e-09\\
17	2.08140659211243e-10\\
18	5.94557553477103e-11\\
19	1.83004057629913e-11\\
20	2.71296685371763e-12\\
21	7.67433189420269e-13\\
22	1.57538026088597e-13\\
23	3.03291059796749e-14\\
24	9.4129891878204e-15\\
25	1.63662795232345e-15\\
26	4.8804954262025e-16\\
27	7.2883396757248e-17\\
28	0\\
};
\addplot [color=black, mark=+, mark options={solid, black}, forget plot]
  table[row sep=crcr]{%
1	0.320177795033949\\
2	0.0607831625511076\\
3	0.0179765320247669\\
4	0.00685001021994036\\
5	0.00249697081901145\\
6	0.000750097837537161\\
7	0.000227256825809431\\
8	6.74559205754175e-05\\
9	2.34300924887438e-05\\
10	9.29772844352821e-06\\
11	2.88203569227046e-06\\
12	7.74332205534703e-07\\
13	2.72103366570637e-07\\
14	1.26533760805635e-07\\
15	4.04807328948159e-08\\
16	1.11355989187578e-08\\
17	3.44376245218127e-09\\
18	1.27116432912355e-09\\
19	3.5238064441467e-10\\
20	9.88403317262992e-11\\
21	3.93836277800925e-11\\
22	1.36472822259021e-11\\
23	3.024034124846e-12\\
24	8.95219682099141e-13\\
25	3.03161287207082e-13\\
26	6.50458574494028e-14\\
27	2.02337053760163e-14\\
28	7.4301783949687e-15\\
29	1.73154291157954e-15\\
30	5.62964739125733e-16\\
31	2.22092941905503e-16\\
32	1.79729515365416e-16\\
33	1.11472125566329e-16\\
34	0\\
35	0\\
36	0\\
};
\addplot [color=black, mark=+, mark options={solid, black}, forget plot]
  table[row sep=crcr]{%
1	0.32017663984044\\
2	0.0607463963159079\\
3	0.0176791869099058\\
4	0.0063966698540766\\
5	0.00274088761461046\\
6	0.00106470651442352\\
7	0.000369691727001161\\
8	0.000131738931388654\\
9	4.85941203064532e-05\\
10	2.06467624249438e-05\\
11	8.7070586255482e-06\\
12	2.99130634838647e-06\\
13	9.95269182676322e-07\\
14	3.43396428340727e-07\\
15	1.51566079875742e-07\\
16	5.34979088756848e-08\\
17	1.64855571041257e-08\\
18	4.9588321026983e-09\\
19	1.95742524960619e-09\\
20	6.66046607303135e-10\\
21	1.89196097953216e-10\\
22	6.19947323680425e-11\\
23	2.34854478668796e-11\\
24	5.74683403868671e-12\\
25	1.49949590361525e-12\\
26	5.38277037033498e-13\\
27	1.42912813636639e-13\\
28	3.12641053610553e-14\\
29	1.15168681052052e-14\\
30	3.15075367318687e-15\\
31	7.99329326126042e-16\\
32	2.806903501365e-16\\
33	1.94138971438579e-16\\
34	1.51678461448806e-16\\
35	9.60131265673129e-17\\
36	0\\
37	0\\
38	0\\
39	0\\
40	0\\
};
\addplot [color=black, mark=+, mark options={solid, black}]
  table[row sep=crcr]{%
1	0.320176547367678\\
2	0.0607424067601834\\
3	0.017630908612853\\
4	0.00618398297616692\\
5	0.00254491496232005\\
6	0.00112739662585988\\
7	0.000454728902940437\\
8	0.000167328441256392\\
9	6.12885364352261e-05\\
10	2.57473783061587e-05\\
11	1.21012153045692e-05\\
12	4.58180394772946e-06\\
13	1.4931028327517e-06\\
14	4.84294058758807e-07\\
15	2.02360518993009e-07\\
16	7.86643623442207e-08\\
17	2.24210387659877e-08\\
18	6.26594536077476e-09\\
19	2.34271487052187e-09\\
20	8.40615832275164e-10\\
21	2.22440084550712e-10\\
22	6.83061794890104e-11\\
23	2.7141798458707e-11\\
24	7.05846094789269e-12\\
25	1.70869760755136e-12\\
26	5.84244541748562e-13\\
27	1.63047340007611e-13\\
28	3.63644848547604e-14\\
29	1.28813417759244e-14\\
30	3.56316246768755e-15\\
31	8.44027265332905e-16\\
32	2.89517656322009e-16\\
33	9.37863714078927e-17\\
};
\addlegendentry{$\sigma_k$}

\addplot [color=black, line width=2.0pt]
  table[row sep=crcr]{%
1	0.00293858063754699\\
2	0.000266006296676968\\
3	2.40794310585473e-05\\
4	2.179719079385e-06\\
5	1.97312604832019e-07\\
6	1.78611383429195e-08\\
7	1.61682657413856e-09\\
8	1.46358430277593e-10\\
9	1.32486628163778e-11\\
10	1.19929590724057e-12\\
};
\addlegendentry{$C\mathrm{e}^{-\gamma_1 k}$}

\addplot [color=black, dashed, line width=2.0pt]
  table[row sep=crcr]{%
1	3.20176547367678\\
2	0.577993630320796\\
3	0.155396477245436\\
4	0.051345536351569\\
5	0.0193547422822505\\
6	0.00801160407470614\\
7	0.00356001785490311\\
8	0.00167328441256392\\
9	0.000823409498573545\\
10	0.000421061796202141\\
11	0.000222489852830906\\
12	0.000120950035276341\\
13	6.74093078731731e-05\\
14	3.84082749936782e-05\\
15	2.23207903534548e-05\\
16	1.32047155511001e-05\\
17	7.93904722149103e-06\\
18	4.84413768327417e-06\\
19	2.99602936720412e-06\\
20	1.87628344363643e-06\\
};
\addlegendentry{$C\mathrm{e}^{-\gamma_2 \sqrt{k}}$}

\end{axis}
\end{tikzpicture}%
  \caption{The singular values of the solutions computed in \nameref{example:DRE_unbdd}, at the final time $T = 0.1$. They increase monotonically, and thus the lower-most line corresponds to $N = 20$ while the top-most corresponds to $N = 65792$.}
  \label{fig:ex03_1}
\end{figure}

\begin{figure}[htbp]
  \centering
%
%
\begin{tikzpicture}

\begin{axis}[%
width=0.761\columnwidth,
height=0.494\columnwidth,
at={(0\columnwidth,0\columnwidth)},
scale only axis,
xmin=0,
xmax=0.1,
xlabel style={font=\color{white!15!black}},
xlabel={$t$},
ymode=log,
ymin=0.01,
ymax=10,
yminorticks=true,
axis background/.style={fill=white},
title style={font=\bfseries},
title={Largest singular value over time},
legend style={legend cell align=left, align=left, draw=white!15!black},
xtick distance= 0.02,
legend pos = south east
]
\addplot [color=black, mark=asterisk, mark options={solid, black}]
  table[row sep=crcr]{%
0.1	0.320176547367678\\
0.05	0.236734975687342\\
0.025	0.178009739838676\\
0.0125	0.131960995042279\\
0.00625	0.0963809825574944\\
0.003125	0.0696910633561737\\
0.0015625	0.0500509541422947\\
0.00078125	0.0357791381757888\\
0.000390625	0.0254963718109634\\
};
\addlegendentry{$\sigma_1$}

\addplot [color=black, line width=2.0pt]
  table[row sep=crcr]{%
0.1	6.52707118360663\\
0.05	3.26353559180332\\
0.025	1.63176779590166\\
0.0125	0.815883897950829\\
0.00625	0.407941948975415\\
0.003125	0.203970974487707\\
0.0015625	0.101985487243854\\
0.00078125	0.0509927436219268\\
0.000390625	0.0254963718109634\\
};
\addlegendentry{$C t^{1-0}$}

\addplot [color=black, dashed, line width=2.0pt]
  table[row sep=crcr]{%
0.1	0.407941948975415\\
0.05	0.288458518450972\\
0.025	0.203970974487707\\
0.0125	0.144229259225486\\
0.00625	0.101985487243854\\
0.003125	0.0721146296127431\\
0.0015625	0.0509927436219268\\
0.00078125	0.0360573148063715\\
0.000390625	0.0254963718109634\\
};
\addlegendentry{$Ct^{1-1/2}$}

\end{axis}
\end{tikzpicture}%
  \caption{The largest singular value of the solution with $N = 65792$ computed in \nameref{example:DRE_unbdd}, plotted over time.}
  \label{fig:ex03_2}
\end{figure}

\subsection{Example 4}\label{example:DRE_too_unbdd}
Let us now consider a situation when the main assumptions are not satisfied. In particular, let us take the same set-up as in \nameref{example:DRE_unbdd} except for the output operator. We now instead take the trace of the normal derivative:
\begin{equation*}
  C x = \int_{\Gamma_T \cap \Gamma_B}{ \bigg( \frac{\partial x}{\partial \nu}\bigg)_{\rvert_{\Gamma}}(s) \dif{s}}.
\end{equation*}
Again by ~\cite[Theorem 8.3]{LioM72}, the map $x \mapsto \big(\frac{\partial x}{\partial \nu}\big)_{\rvert_{\Gamma}}$ belongs to $\cL(H^{3/2}(\Omega), L^2(\Gamma))$ and hence the map $CA^{-\alpha}$ is bounded for $\alpha = 3/4 + \epsilon$, $\epsilon > 0$. Since $\alpha  >  1/2$, \cref{assumption:C} is not satisfied, and we can in fact not show the existence of a solution $P \in \LHH$.

This is reflected in the results shown in \cref{fig:ex04}. We have discretized the problem in the same way as previously, and we plot the singular values for the different discretizations like in \cref{fig:ex01,fig:ex02_1}. In contrast to the previous results, we now see that the singular values keep increasing as we refine the discretization, demonstrating that the singular values of the exact solution are infinite. Thus, while the singular values of a single discretized matrix-valued equation seem to decay exponentially, since the underlying problem is not well posed these ``approximations'' are nevertheless worthless.

\begin{figure}[htbp]
  \centering
%
%
\begin{tikzpicture}

\begin{axis}[%
width=0.761\columnwidth,
height=0.494\columnwidth,
at={(0\columnwidth,0\columnwidth)},
scale only axis,
xmin=0,
xmax=35,
xlabel style={font=\color{white!15!black}},
xlabel={$k$},
ymode=log,
ymin=1e-16,
ymax=10000,
yminorticks=true,
ylabel style={font=\color{white!15!black}},
ylabel={$\sigma_k$},
axis background/.style={fill=white},
title style={font=\bfseries},
title={Singular values for different discretizations},
legend style={legend cell align=left, align=left, draw=white!15!black},
xtick distance= 5,
legend pos = north east
]
\addplot [color=black, mark=+, mark options={solid, black}, forget plot]
  table[row sep=crcr]{%
1	11.4717057472521\\
2	0.799754750153594\\
3	0.0909037893123634\\
4	0.013414835275718\\
5	0.00179040074738702\\
6	0.000175484054225928\\
7	1.79494822907199e-05\\
8	1.23989880243374e-06\\
9	9.47697593147386e-08\\
10	5.01885581526649e-09\\
11	2.31650450192585e-10\\
12	1.19477028881194e-11\\
13	5.8570884047532e-13\\
14	1.92107562298344e-14\\
15	8.60484603929157e-15\\
16	0\\
17	0\\
};
\addplot [color=black, mark=+, mark options={solid, black}, forget plot]
  table[row sep=crcr]{%
1	23.0721993861899\\
2	1.35830485078011\\
3	0.16392549404135\\
4	0.0293779387640003\\
5	0.0085448144134859\\
6	0.00106969602345861\\
7	0.000264674609892087\\
8	5.33061100812936e-05\\
9	9.34111396534168e-06\\
10	1.52304200092177e-06\\
11	2.31526160167999e-07\\
12	2.79726084740141e-08\\
13	4.10321889674002e-09\\
14	6.19407475727144e-10\\
15	5.76014703612864e-11\\
16	4.77760871293549e-12\\
17	7.26963753997415e-13\\
18	3.14332457382268e-13\\
19	7.98601963166361e-14\\
20	1.88186497453238e-14\\
21	1.37971322392975e-14\\
22	0\\
};
\addplot [color=black, mark=+, mark options={solid, black}, forget plot]
  table[row sep=crcr]{%
1	46.2639668131427\\
2	2.60268975186027\\
3	0.261047045223809\\
4	0.054307508942811\\
5	0.011286843447854\\
6	0.00243831780709487\\
7	0.000592055894368569\\
8	0.000199382001856706\\
9	4.69807716219495e-05\\
10	1.88652439601871e-05\\
11	3.37603732279068e-06\\
12	1.16383119084143e-06\\
13	2.09867007629628e-07\\
14	4.31660910153024e-08\\
15	1.09565205020286e-08\\
16	2.22110097664254e-09\\
17	4.2076169568693e-10\\
18	8.02097150858657e-11\\
19	1.76841112410987e-11\\
20	3.48302198659523e-12\\
21	7.07969958324519e-13\\
22	3.55196684151481e-13\\
23	1.02146618489726e-13\\
24	1.99019671907811e-14\\
25	0\\
26	0\\
};
\addplot [color=black, mark=+, mark options={solid, black}, forget plot]
  table[row sep=crcr]{%
1	92.6695400374594\\
2	5.08169170718399\\
3	0.457922287287119\\
4	0.0876889643115113\\
5	0.019103759211827\\
6	0.00431151314277681\\
7	0.000962249963957456\\
8	0.000275732677663474\\
9	8.10034927104983e-05\\
10	2.94694103467946e-05\\
11	1.21971985448925e-05\\
12	4.04280499775716e-06\\
13	1.34677948184825e-06\\
14	4.7841071114727e-07\\
15	1.25176074977589e-07\\
16	4.61762234115846e-08\\
17	1.33662185341555e-08\\
18	3.48383284448941e-09\\
19	1.19743232226415e-09\\
20	2.33202359957786e-10\\
21	7.85410066052674e-11\\
22	1.63248159462824e-11\\
23	5.38421856690597e-12\\
24	1.19743569252899e-12\\
25	3.71503290545389e-13\\
26	6.12468626831138e-14\\
27	0\\
28	0\\
};
\addplot [color=black, mark=+, mark options={solid, black}, forget plot]
  table[row sep=crcr]{%
1	185.497527647469\\
2	10.0348342111497\\
3	0.850536884154935\\
4	0.147381864673296\\
5	0.0313898949539237\\
6	0.00762963466995392\\
7	0.00167306474685075\\
8	0.000455443960352732\\
9	0.000125998602041506\\
10	4.31200500563698e-05\\
11	1.62914368926325e-05\\
12	5.93151269839907e-06\\
13	2.36207473271038e-06\\
14	9.95589495923865e-07\\
15	3.98839152348003e-07\\
16	1.42458817695091e-07\\
17	5.46641901459763e-08\\
18	2.15579696992407e-08\\
19	7.79947603742488e-09\\
20	3.08038305411884e-09\\
21	1.1551507525268e-09\\
22	3.4802726439259e-10\\
23	1.19208737169967e-10\\
24	3.95830301324255e-11\\
25	1.0244041903588e-11\\
26	3.42647910975405e-12\\
27	1.18075711003118e-12\\
28	9.23904331649873e-13\\
29	4.9109874636866e-13\\
30	3.15939466290242e-13\\
31	1.77639829396138e-13\\
32	1.35669178530493e-13\\
33	0\\
};
\addplot [color=black, mark=+, mark options={solid, black}, forget plot]
  table[row sep=crcr]{%
1	370.980085836899\\
2	19.8239438354706\\
3	1.57896332249553\\
4	0.22904797158929\\
5	0.033841264866239\\
6	0.00695866454751729\\
7	0.00139933682625873\\
8	0.000313769553864065\\
9	8.48934830306415e-05\\
10	2.39643535255505e-05\\
11	7.09111830437647e-06\\
12	2.30155662088207e-06\\
13	8.19898293257311e-07\\
14	3.07842961319744e-07\\
15	1.19448855044169e-07\\
16	4.96775335642064e-08\\
17	2.00065098044534e-08\\
18	7.17730499519336e-09\\
19	2.64430078376057e-09\\
20	1.06326634298727e-09\\
21	3.97431132955898e-10\\
22	1.38372068694536e-10\\
23	4.7733057136067e-11\\
24	1.52959743879981e-11\\
25	5.11262573151448e-12\\
26	1.58249814559657e-12\\
27	7.46180505646128e-13\\
28	3.7138939538205e-13\\
29	3.36240405221214e-13\\
30	2.84823094183561e-13\\
31	1.19669511355762e-13\\
32	0\\
33	0\\
34	0\\
35	0\\
};
\addplot [color=black, mark=+, mark options={solid, black}]
  table[row sep=crcr]{%
1	557.582750518393\\
2	13.6731332864346\\
3	0.658783988393988\\
4	0.0435190762198005\\
5	0.00441940070335371\\
6	0.000613264422774149\\
7	0.000106711343871976\\
8	2.23539089060813e-05\\
9	5.45345997145039e-06\\
10	1.49468703969408e-06\\
11	4.48721766805667e-07\\
12	1.45205942709836e-07\\
13	4.99666137319434e-08\\
14	1.8083078322106e-08\\
15	6.9597347289328e-09\\
16	2.92705948212967e-09\\
17	1.19205931443441e-09\\
18	4.30061880977175e-10\\
19	1.55162778927193e-10\\
20	6.1021518461204e-11\\
21	2.23232313884682e-11\\
22	7.68412498484244e-12\\
23	2.72792951449892e-12\\
24	6.46468044239906e-13\\
25	6.12212804610183e-13\\
26	0\\
};
\addlegendentry{$\sigma_k$}

\addplot [color=black, line width=2.0pt]
  table[row sep=crcr]{%
1	0.120859403703835\\
2	0.0127330632300655\\
3	0.0013414835275718\\
4	0.000141331117440562\\
5	1.48898472075563e-05\\
6	1.5687100893235e-06\\
7	1.65270422862131e-07\\
8	1.74119570333143e-08\\
9	1.8344253162763e-09\\
10	1.93264676369057e-10\\
};
\addlegendentry{$C\mathrm{e}^{-\gamma_1 k}$}

\addplot [color=black, dashed, line width=2.0pt]
  table[row sep=crcr]{%
1	5575.82750518393\\
2	117.651142313081\\
3	6.09222872280815\\
4	0.502087998070038\\
5	0.0556877762181334\\
6	0.00762703659780119\\
7	0.0012256708801754\\
8	0.000223539089060813\\
9	4.52116493151203e-05\\
10	9.97128869919425e-06\\
11	2.36774770183451e-06\\
12	5.99395640649863e-07\\
13	1.60499915416618e-07\\
14	4.51703567754339e-08\\
15	1.32913923801708e-08\\
16	4.07118521384828e-09\\
17	1.2932825618916e-09\\
18	4.24727914712156e-10\\
19	1.43808514645304e-10\\
20	5.00819155333455e-11\\
};
\addlegendentry{$C\mathrm{e}^{-\gamma_2 \sqrt{k}}$}

\end{axis}
\end{tikzpicture}%
  \caption{The singular values of the solutions computed in \nameref{example:DRE_too_unbdd}, at the final time $T = 0.1$. They increase (roughly) monotonically, and thus the lower-most line corresponds to $N = 20$ while the top-most corresponds to $N = 65792$. Because the underlying problem is not well-posed, the discretized solutions increase without bound.}
  \label{fig:ex04}
\end{figure}
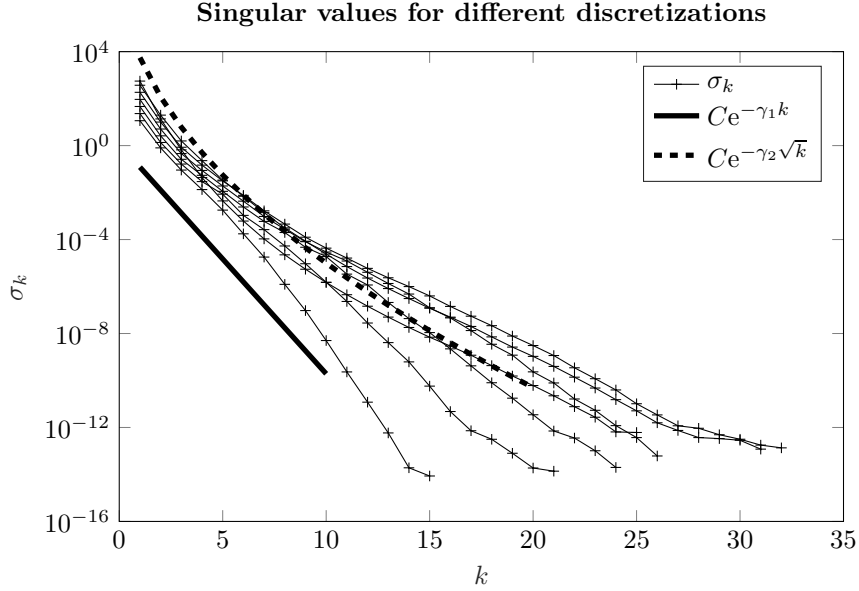

\subsection{Example 5}\label{example:DRE_too_unbdd_H1}
The situation in the previous Example holds when we use $H = L^2(\Omega)$. By instead selecting a smaller state space $H$, we decrease the value of $\alpha$. With $H = \{ x \in H^1(\Omega) \;;\; x_{\rvert_{\Gamma_L}} = 0\}$ and the same operator $C$ we again get $\alpha = 1/4 + \epsilon$. Since we simultaneously increase $\beta$ by $1/2$, we set $B = 0$ in this example to comply with \cref{assumption:B}. 

We note that we now consider the operator $A$ as restricted to $H$ instead of an operator on $L^2(\Omega)$. It still generates an analytic semigroup and \cref{assumption:A} is satisfied. Since the finite-element discretization of the problem is no longer based on $a(u,v)$ but on the corresponding inner product defined on $H^1$, the resulting problem is similar to a biharmonic equation. This imposes extra regularity requirements on the standard conforming finite element spaces, requiring a high number of nodes~\cite[p.~286]{BreS08}. In order to avoid this, in this example we employ instead the nonconforming Morley elements~\cite{Mor68,BreS08}.

The results are shown in \cref{fig:ex05}. We see that since $\alpha$ is now again less than $1/2$, the singular values behave much like in the previous examples.

\begin{figure}[htbp]
  \centering
%
%
\begin{tikzpicture}

\begin{axis}[%
width=0.761\columnwidth,
height=0.494\columnwidth,
at={(0\columnwidth,0\columnwidth)},
scale only axis,
xmin=0,
xmax=35,
xlabel style={font=\color{white!15!black}},
xlabel={$k$},
ymode=log,
ymin=1e-16,
ymax=100,
yminorticks=true,
ylabel style={font=\color{white!15!black}},
ylabel={$\sigma_k$},
axis background/.style={fill=white},
title style={font=\bfseries},
title={Singular values for different discretizations},
legend style={legend cell align=left, align=left, draw=white!15!black},
xtick distance= 5,
legend pos = north east
]
\addplot [color=black, mark=+, mark options={solid, black}, forget plot]
  table[row sep=crcr]{%
1	0.555294017375873\\
2	0.0729235524991152\\
3	0.00297943970692152\\
4	0.000238929356964017\\
5	6.0528474688259e-06\\
6	6.92448515278281e-08\\
7	1.49656351639972e-09\\
8	8.39715610392641e-12\\
9	1.19967844762473e-13\\
10	2.46780599046065e-15\\
11	1.56947850158348e-15\\
12	1.11464959840724e-15\\
13	9.86166997542847e-16\\
14	7.67041868016087e-16\\
15	2.32044848316472e-16\\
16	2.11368895034824e-16\\
17	0\\
18	0\\
19	0\\
};
\addplot [color=black, mark=+, mark options={solid, black}, forget plot]
  table[row sep=crcr]{%
1	0.455138637401064\\
2	0.104890284949332\\
3	0.0117133661994944\\
4	0.00129577419380907\\
5	0.000195456034242621\\
6	1.68316940902e-05\\
7	1.21185384334485e-06\\
8	6.60033699798702e-08\\
9	5.30765841674026e-09\\
10	3.48376172058214e-10\\
11	1.37962333837058e-11\\
12	4.67080712331264e-13\\
13	2.50852641529515e-14\\
14	9.47723261205858e-16\\
15	3.92713516819526e-16\\
16	2.95073716092033e-16\\
17	1.89007435793871e-16\\
18	0\\
19	0\\
};
\addplot [color=black, mark=+, mark options={solid, black}, forget plot]
  table[row sep=crcr]{%
1	0.427975324912588\\
2	0.0887988906990756\\
3	0.0244625232552311\\
4	0.00418751360900922\\
5	0.000696478586070297\\
6	0.000135177204621318\\
7	2.18132733036023e-05\\
8	3.82226185342837e-06\\
9	5.01938245251342e-07\\
10	5.84553022256673e-08\\
11	7.75368990712022e-09\\
12	1.54601204965203e-09\\
13	1.58601892654422e-10\\
14	1.93505767601567e-11\\
15	1.49007280435234e-12\\
16	1.56189229643339e-13\\
17	2.40615910529829e-14\\
18	2.20765682239026e-15\\
};
\addplot [color=black, mark=+, mark options={solid, black}, forget plot]
  table[row sep=crcr]{%
1	0.42034455941051\\
2	0.0762241000897388\\
3	0.0276301074769854\\
4	0.00797132928097739\\
5	0.00200015752734455\\
6	0.000402261292328604\\
7	9.46655504560369e-05\\
8	2.34271215548452e-05\\
9	6.55827778149134e-06\\
10	1.31777668407635e-06\\
11	2.46152973732126e-07\\
12	6.2585053167003e-08\\
13	8.98876442509245e-09\\
14	2.17787645518547e-09\\
15	6.09121905644656e-10\\
16	8.46407624433993e-11\\
17	2.00411428783177e-11\\
18	2.71353910091316e-12\\
19	4.37142855674937e-13\\
20	7.56670373880992e-14\\
21	1.87274211771866e-14\\
22	3.3494888201219e-15\\
23	3.95866431846824e-16\\
24	0\\
};
\addplot [color=black, mark=+, mark options={solid, black}, forget plot]
  table[row sep=crcr]{%
1	0.417812916071409\\
2	0.0734284710188245\\
3	0.0238873646067643\\
4	0.0101828111621939\\
5	0.00351997088132485\\
6	0.00102064362947924\\
7	0.000247026814361128\\
8	6.87116782498514e-05\\
9	2.19778678984011e-05\\
10	7.83931647267251e-06\\
11	2.44175311079155e-06\\
12	5.46844044281222e-07\\
13	1.56434432097441e-07\\
14	5.39169330287824e-08\\
15	1.13006254906876e-08\\
16	2.71130835457256e-09\\
17	8.29455835475732e-10\\
18	2.47128054593402e-10\\
19	6.28260182686681e-11\\
20	1.65473875107124e-11\\
21	3.6946884872647e-12\\
22	9.06605347569163e-13\\
23	1.88472150993802e-13\\
24	5.36830655070321e-14\\
25	1.33554701855771e-14\\
26	4.62897391089118e-15\\
27	3.55482160684098e-15\\
28	1.82810494671517e-15\\
29	1.67835318033266e-15\\
30	1.48437507789638e-15\\
31	9.92080792262864e-16\\
32	7.57659890540564e-16\\
33	6.31599594250169e-16\\
34	0\\
35	0\\
36	0\\
};
\addplot [color=black, mark=+, mark options={solid, black}, forget plot]
  table[row sep=crcr]{%
1	0.41694807241364\\
2	0.073175312804879\\
3	0.021904379894243\\
4	0.00954306920228542\\
5	0.00448028185253262\\
6	0.00173313880772561\\
7	0.000549006118841997\\
8	0.000158967930988088\\
9	5.10560243224564e-05\\
10	1.91991590290806e-05\\
11	7.88519685091742e-06\\
12	3.10236624845453e-06\\
13	9.91225621011566e-07\\
14	2.8741512848679e-07\\
15	1.07869737306704e-07\\
16	4.1442934287402e-08\\
17	1.20803540700239e-08\\
18	3.59438178159864e-09\\
19	1.04855564169818e-09\\
20	3.18778617442694e-10\\
21	1.29892182180376e-10\\
22	4.52199479105095e-11\\
23	1.23079651423864e-11\\
24	3.86542962515197e-12\\
25	1.29385286713891e-12\\
26	3.35571915504012e-13\\
27	1.05422606347886e-13\\
28	3.54802284495602e-14\\
29	8.75776375884729e-15\\
30	7.48043897279628e-15\\
31	5.09540177690301e-15\\
32	4.75134145369671e-15\\
33	1.37638611811234e-15\\
34	1.01235078649437e-15\\
35	7.27366576724722e-16\\
36	0\\
};
\addplot [color=black, mark=+, mark options={solid, black}]
  table[row sep=crcr]{%
1	0.416668176727274\\
2	0.0731743233262713\\
3	0.0216471623836172\\
4	0.00860365202946289\\
5	0.00436727162907277\\
6	0.00218071859281525\\
7	0.000885920348019039\\
8	0.00031092053638486\\
9	0.000104472539037054\\
10	3.71700791583258e-05\\
11	1.48795232051584e-05\\
12	6.31330244157523e-06\\
13	2.42111902600726e-06\\
14	8.47058597787407e-07\\
15	3.04345943258369e-07\\
16	1.22634632414343e-07\\
17	4.74153313110577e-08\\
18	1.46577196847312e-08\\
19	4.81193811133512e-09\\
20	1.74628853578247e-09\\
21	5.37795958888388e-10\\
22	1.73636376000195e-10\\
23	6.57453678400171e-11\\
24	2.07054144841669e-11\\
25	5.54293458236477e-12\\
26	1.76663400071693e-12\\
27	4.97258892161108e-13\\
28	1.38499559342657e-13\\
29	4.87874187202968e-14\\
30	1.26420345439224e-14\\
31	0\\
32	0\\
};
\addlegendentry{$\sigma_k$}

\addplot [color=black, line width=2.0pt]
  table[row sep=crcr]{%
1	0.00419216628691956\\
2	0.000316485638729457\\
3	2.38929356964017e-05\\
4	1.80378603744605e-06\\
5	1.36175985664889e-07\\
6	1.02805425293457e-08\\
7	7.76124763714019e-10\\
8	5.85931770751089e-11\\
9	4.42346457717216e-12\\
10	3.33947395281442e-13\\
};
\addlegendentry{$C\mathrm{e}^{-\gamma_1 k}$}

\addplot [color=black, dashed, line width=2.0pt]
  table[row sep=crcr]{%
1	4.16668176727274\\
2	0.815389193563967\\
3	0.233222817548471\\
4	0.0811895566577041\\
5	0.032044642465412\\
6	0.0138274559865894\\
7	0.00638378210926167\\
8	0.0031092053638486\\
9	0.0015820128530212\\
10	0.000834964712954767\\
11	0.000454663136396787\\
12	0.000254367476146267\\
13	0.000145727534854374\\
14	8.52630129256134e-05\\
15	5.08341046553554e-05\\
16	3.08261896006643e-05\\
17	1.89833784058568e-05\\
18	1.18558819161325e-05\\
19	7.50064330318332e-06\\
20	4.80209078373975e-06\\
};
\addlegendentry{$C\mathrm{e}^{-\gamma_2 \sqrt{k}}$}

\end{axis}
\end{tikzpicture}%
  \caption{The singular values of the solutions computed in \nameref{example:DRE_too_unbdd_H1}, at the final time $T = 0.1$. They increase monotonically, and thus the lower-most line corresponds to $N = 20$ while the top-most corresponds to $N = 65792$.}
  \label{fig:ex05}
\end{figure}

\section{Conclusions}
We have proved bounds for the singular values $\sigma_k$ of the solutions to DLEs and DREs of the form $\sigma_k \le M\exp{-\gamma \sqrt{k}}$, extending previous results on algebraic equations to the differential case. This is important, since utilizing the property of low numerical rank is a critical feature in numerical methods for these problems in the large-scale setting. If low numerical rank, i.e.\ a sufficiently rapid decay of the singular values, can not be guaranteed, these methods never finish, or fail outright. The current work is thus a step on the way to provide practical criteria for when this is to be expected. We say ``a step on the way'' because while we have given conditions for when exponential square-root decay is to be expected, we have not indicated how large the constant multiplier in the bound can be. A large value could mean that the numerical rank is too large to be useful in a practical application, even though the decay is $\mathcal{O}(\exp{-\gamma \sqrt{k}})$. However, the size of this constant depends strongly on the properties of the operators $A$ and $C$, and providing a generally meaningful bound is difficult with current techniques. We therefore leave this question open for future research, but note that the constants arising in our numerical experiments are all of moderate size.

A further interesting unexplored question is how the singular values of the solutions to the spatially discretized matrix-valued problems relate to those of the operator-valued solutions. As noted in the numerical experiments, one often observes exponential decay in the discretized case. When the discretization is refined, the decay rate deteriorates and eventually tends to the exponential square-root bound. The form of this decrease is, however, unclear. While it can be argued that the discretized equations are only steps on the way towards the non-discretized goal (and the author does argue thus), in practical computations we are of course always in the matrix-valued situation. Analysing also this case and providing a connection between the decay rate and the discretization level is therefore both highly interesting and important, but clearly requires a different approach.

\section{Acknowledgements}  
The author is grateful to Mark Opmeer for providing several helpful references. The careful reading and constructive comments from the anonymous referees also led to a greatly improved manuscript.

\bibliographystyle{siamplain}

\begin{thebibliography}{10}

\bibitem{AboFIJ03}
{\sc H.~Abou-Kandil, G.~Freiling, V.~Ionescu, and G.~Jank}, {\em Matrix
  {R}iccati Equations in Control and Systems Theory}, Birkh{\"a}user, Basel,
  Switzerland, 2003.

\bibitem{AntSZ02}
{\sc A.~C. Antoulas, D.~C. Sorensen, and Y.~Zhou}, {\em On the decay rate of
  {H}ankel singular values and related issues}, Syst. Cont. Lett., 46 (2002),
  pp.~323--342, \url{https://doi.org/10.1016/S0167-6911(02)00147-0}.

\bibitem{BasB95}
{\sc T.~Ba\c{s}ar and P.~Bernhard}, {\em {$H^{\infty}$}-optimal control and
  related minimax design problems}, Systems \& Control: Foundations \&
  Applications, Birkh\"auser Boston, Inc., Boston, MA, second~ed., 1995,
  \url{https://doi.org/10.1007/978-0-8176-4757-5}.
\newblock A dynamic game approach.

\bibitem{BakES15}
{\sc J.~Baker, M.~Embree, and J.~Sabino}, {\em Fast singular value decay for
  {L}yapunov solutions with nonnormal coefficients}, {SIAM} J. Matrix Anal.
  Appl., 36 (2015), pp.~656--668, \url{https://doi.org/10.1137/140993867}.

\bibitem{morBauBF14}
{\sc U.~Baur, P.~Benner, and L.~Feng}, {\em Model order reduction for linear
  and nonlinear systems: A system-theoretic perspective}, Arch. Comput. Methods
  Eng., 21 (2014), pp.~331--358,
  \url{https://doi.org/10.1007/s11831-014-9111-2}.

\bibitem{BenB13}
{\sc P.~Benner and T.~Breiten}, {\em Low rank methods for a class of
  generalized {L}yapunov equations and related issues}, Numerische Mathematik,
  124 (2013), pp.~441--470, \url{https://doi.org/10.1007/s00211-013-0521-0}.

\bibitem{BenB16}
{\sc P.~Benner and Z.~Bujanovi\'c}, {\em On the solution of large-scale
  algebraic {R}iccati equations by using low-dimensional invariant subspaces},
  Linear Algebra Appl., 488 (2016), pp.~430--459,
  \url{https://doi.org/10.1016/j.laa.2015.09.027}.

\bibitem{morBenKS16}
{\sc P.~Benner, P.~K{\"u}rschner, and J.~Saak}, {\em Frequency-limited balanced
  truncation with low-rank approximations}, {SIAM} J. Sci. Comput., 38 (2016),
  pp.~A471--A499, \url{https://doi.org/10.1137/15M1030911}.

\bibitem{BenLP08}
{\sc P.~Benner, J.-R. Li, and T.~Penzl}, {\em Numerical solution of large-scale
  {L}yapunov equations, {R}iccati equations, and linear-quadratic optimal
  control problems}, Numer. Lin. Alg. Appl., 15 (2008), pp.~755--777,
  \url{https://doi.org/10.1002/nla.622}.

\bibitem{BenDDetal07}
{\sc A.~Bensoussan, G.~{Da Prato}, M.~C. Delfour, and S.~K. Mitter}, {\em
  Representation and Control of Infinite Dimensional Systems}, Systems \&
  Control: Foundations \& Applications, Birkh{\"a}user, Boston, MA, second~ed.,
  2007.

\bibitem{BreS08}
{\sc S.~C. Brenner and L.~R. Scott}, {\em The mathematical theory of finite
  element methods}, vol.~15 of Texts in Applied Mathematics, Springer, New
  York, third~ed., 2008, \url{https://doi.org/10.1007/978-0-387-75934-0}.

\bibitem{CouH53}
{\sc R.~Courant and D.~Hilbert}, {\em Methods of mathematical physics. {V}ol.
  {I}}, Interscience Publishers, Inc., New York, N.Y., 1953.

\bibitem{CurS01}
{\sc R.~F. Curtain and A.~J. Sasane}, {\em Compactness and nuclearity of the
  {H}ankel operator and internal stability of infinite-dimensional state linear
  systems}, Internat. J. Control, 74 (2001), pp.~1260--1270,
  \url{https://doi.org/10.1080/00207170110061059}.

\bibitem{Fan51}
{\sc K.~Fan}, {\em Maximum properties and inequalities for the eigenvalues of
  completely continuous operators}, Proc. Nat. Acad. Sci. U.S.A., 37 (1951),
  pp.~760--766.

\bibitem{morGawJ90}
{\sc W.~Gawronski and J.-N. Juang}, {\em Model reduction in limited time and
  frequency intervals}, Int. J. Syst. Sci., 21 (1990), pp.~349--376,
  \url{https://doi.org/10.1080/00207729008910366}.

\bibitem{GruK14}
{\sc L.~Grubi\v{s}i\'{c} and D.~Kressner}, {\em On the eigenvalue decay of
  solutions to operator {L}yapunov equations}, Syst. Cont. Lett., 73 (2014),
  pp.~42--47, \url{https://doi.org/10.1016/j.sysconle.2014.09.006}.

\bibitem{Hec12}
{\sc F.~Hecht}, {\em New development in freefem++}, J. Numer. Math., 20 (2012),
  pp.~251--265.

\bibitem{IchK99}
{\sc A.~Ichikawa and H.~Katayama}, {\em Remarks on the time-varying {$H\sb
  \infty$} {R}iccati equations}, Syst. Cont. Lett., 37 (1999), pp.~335--345.

\bibitem{morKue17}
{\sc P.~K{\"u}rschner}, {\em Balanced truncation model order reduction in
  limited time intervals for large systems}, arXiv e-print 1707.02839, Cornell
  University, 2017, \url{http://arxiv.org/abs/1707.02839}.
\newblock Math.NA.

\bibitem{LanMS15}
{\sc N.~Lang, H.~Mena, and J.~Saak}, {\em On the benefits of the {$LDL^T$}
  factorization for large-scale differential matrix equation solvers}, Linear
  Algebra Appl., 480 (2015), pp.~44--71,
  \url{https://doi.org/10.1016/j.laa.2015.04.006}.

\bibitem{LasT00}
{\sc I.~Lasiecka and R.~Triggiani}, {\em Control Theory for Partial
  Differential Equations: Continuous and Approximation Theories {I}. {A}bstract
  Parabolic Systems}, Cambridge University Press, Cambridge, UK, 2000.

\bibitem{LasT00b}
{\sc I.~Lasiecka and R.~Triggiani}, {\em Control theory for partial
  differential equations: Continuous and approximation theories {II}.
  {A}bstract hyperbolic-like systems over a finite time horizon}, in
  Encyclopedia of Mathematics and its Applications, vol.~75, Cambridge
  University Press, Cambridge, 2000, pp.~645--1067.

\bibitem{LiW02}
{\sc J.-R. Li and J.~White}, {\em Low rank solution of {L}yapunov equations},
  {SIAM} J. Matrix Anal. Appl., 24 (2002), pp.~260--280,
  \url{https://doi.org/10.1137/S0895479801384937}.

\bibitem{LioM72}
{\sc J.-L. Lions and E.~Magenes}, {\em Non-homogeneous boundary value problems
  and applications. {V}ol. {I}}, Springer-Verlag, New York-Heidelberg, 1972.
\newblock Translated from the French by P. Kenneth, Die Grundlehren der
  mathematischen Wissenschaften, Band 181.

\bibitem{LunB92}
{\sc J.~Lund and K.~L. Bowers}, {\em Sinc {M}ethods for {Q}uadrature and
  {D}ifferential {E}quations}, Society for Industrial and Applied Mathematics
  (SIAM), Philadelphia, PA, 1992,
  \url{https://doi.org/10.1137/1.9781611971637}.

\bibitem{MalPS18}
{\sc A.~M{\aa}lqvist, A.~Persson, and T.~Stillfjord}, {\em Multiscale
  differential {R}iccati equations for linear quadratic regulator problems},
  ArXiv e-prints,  (2018), \url{https://arxiv.org/abs/1706.04380}.
\newblock To appear in {SIAM} J.\ Sci.\ Comput.

\bibitem{Mik02}
{\sc K.~M. Mikkola}, {\em Infinite-dimensional linear systems, optimal control
  and algebraic {R}iccati equations}, {D}issertation, Helsinki University of
  Technology, Helsinki, Finland, Oct. 2002,
  \url{http://lib.tkk.fi/Diss/2002/isbn9512260794/}.

\bibitem{Mor68}
{\sc L.~S.~D. Morley}, {\em The triangular equilibrium element in the solution
  of plate bending problems}, Aeronaut. Quart., 19 (1968), pp.~149--169.

\bibitem{Opm15}
{\sc M.~Opmeer}, {\em Decay of singular values of the {G}ramians of
  infinite-dimensional systems}, in Proceedings 2015 European Control
  Conference (ECC), Linz, Austria, 2015, {IEEE}, pp.~1183--1188,
  \url{https://doi.org/10.1109/ECC.2015.7330700}.

\bibitem{Paz83}
{\sc A.~Pazy}, {\em Semigroups of linear operators and applications to partial
  differential equations.}, vol.~44 of Applied Mathematical Sciences,
  Springer-Verlag, New York etc., 1983.

\bibitem{Pen00}
{\sc T.~Penzl}, {\em Eigenvalue decay bounds for solutions of {L}yapunov
  equations: the symmetric case}, Syst. Cont. Lett., 40 (2000), pp.~139--144,
  \url{https://doi.org/10.1016/S0167-6911(00)00010-4}.

\bibitem{PetUS00}
{\sc I.~R. Petersen, V.~A. Ugrinovskii, and A.~V. Savkin}, {\em Robust Control
  Design Using $H^{\infty}$ Methods}, Springer-Verlag, London, UK, 2000.

\bibitem{Sal87}
{\sc D.~Salamon}, {\em Infinite-dimensional linear systems with unbounded
  control and observation: a functional analytic approach}, Trans. Amer. Math.
  Soc., 300 (1987), pp.~383--431, \url{https://doi.org/10.2307/2000351}.

\bibitem{SorZ02}
{\sc D.~C. Sorensen and Y.~Zhou}, {\em Bounds on eigenvalue decay rates and
  sensitivity of solutions to {L}yapunov equations}, Tech. Report TR02-07,
  Dept. of Comp. Appl. Math., Rice University, Houston, TX, June 2002.
\newblock Available online from \url{%
  http://www.caam.rice.edu/caam/trs/tr02.html\#TR02-07}.

\bibitem{Sta05}
{\sc O.~Staffans}, {\em Well-posed linear systems}, vol.~103 of Encyclopedia of
  Mathematics and its Applications, Cambridge University Press, Cambridge,
  2005, \url{https://doi.org/10.1017/CBO9780511543197}.

\bibitem{Ste73}
{\sc F.~Stenger}, {\em Integration {F}ormulae {B}ased on the {T}rapezoidal
  {F}ormula}, J. Inst. Math. Appl., 12 (1973), pp.~103--114.

\bibitem{Ste93}
{\sc F.~Stenger}, {\em Numerical {M}ethods {B}ased on {S}inc and {A}nalytic
  {F}unctions}, vol.~20 of Springer Series in Computational Mathematics,
  Springer-Verlag, New York, 1993,
  \url{https://doi.org/10.1007/978-1-4612-2706-9}.

\bibitem{Sti15}
{\sc T.~Stillfjord}, {\em Low-rank second-order splitting of large-scale
  differential {R}iccati equations}, {IEEE} Trans. Autom. Control, 60 (2015),
  pp.~2791--2796, \url{https://doi.org/10.1109/TAC.2015.2398889}.

\bibitem{Sti17}
{\sc T.~Stillfjord}, {\em Adaptive high-order splitting schemes for large-scale
  differential {R}iccati equations}, Numer. Algorithms,  (2017),
  \url{https://doi.org/10.1007/s11075-017-0416-8}.

\bibitem{TucW09}
{\sc M.~Tucsnak and G.~Weiss}, {\em Observation and control for operator
  semigroups}, Birkh\"auser Advanced Texts: Basler Lehrb\"ucher. [Birkh\"auser
  Advanced Texts: Basel Textbooks], Birkh\"auser Verlag, Basel, 2009,
  \url{https://doi.org/10.1007/978-3-7643-8994-9}.

\end{thebibliography}

\end{document}